\documentclass[a4paper, 12pt]{article}
\usepackage[utf8]{inputenc}
\usepackage{a4}
\usepackage[T1]{fontenc}
\usepackage[english]{babel}
\usepackage{amssymb,amsmath,amsthm}
\usepackage{enumerate}
\usepackage{graphicx}
\usepackage{epsfig}
\usepackage{comment}
\usepackage{xspace}
\usepackage{float}
\usepackage{multirow}
\usepackage{extarrows}
\usepackage{mathtools}
\usepackage{color}
\usepackage{nomencl}
\usepackage{nicefrac}
\makenomenclature
\newcommand{\N}{\ensuremath{\mathbb{N}}\xspace}

\newcommand{\R}{\ensuremath{\mathbb{R}}\xspace}

\newcommand{\eps}{\epsilon}

\renewcommand{\epsilon}{\varepsilon}

\newcommand{\mres}{\mathbin{\vrule height 1.6ex depth 0pt width 0.13ex\vrule height 0.13ex depth 0pt width 1.3ex}} 
\newcommand{\hess}{\nabla ^2}
\newcommand{\palpha}{\partial_\alpha}
\newcommand{\udel}{u_\delta}
\newcommand{\into}{\intop_{\Omega}}
\newcommand{\intb}{\intop_{B_{2r}(x_0)}}
\newcommand{\tilv}{\hat{v}_k}
\newcommand{\dx}{\, dx}
\newcommand{\dz}{\, dz}

\def\Yint#1{\mathchoice
    {\YYint\displaystyle\textstyle{#1}}%
    {\YYint\textstyle\scriptstyle{#1}}%
    {\YYint\scriptstyle\scriptscriptstyle{#1}}%
    {\YYint\scriptscriptstyle\scriptscriptstyle{#1}}%
      \!\int}
\def\YYint#1#2#3{{\setbox0=\hbox{$#1{#2#3}{\int}$}
    \vcenter{\hbox{$#2#3$}}\kern-.72\wd0}}

\def\mint{\;\;\Yint -}

\def\Zint#1{\mathchoice
    {\ZZint\displaystyle\textstyle{#1}}%
    {\ZZint\textstyle\scriptstyle{#1}}%
    {\ZZint\scriptstyle\scriptscriptstyle{#1}}%
    {\ZZint\scriptscriptstyle\scriptscriptstyle{#1}}%
      \!\int}
\def\ZZint#1#2#3{{\setbox0=\hbox{$#1{#2#3}{\int}$}
    \vcenter{\hbox{$#2#3$}}\kern-.52\wd0}}

\def\txtmint{\;\;\Zint -}

\begin{document}

\numberwithin{equation}{section}
\newtheoremstyle{break}{15pt}{15pt}{\itshape}{}{\bfseries}{}{\newline}{}
\theoremstyle{break}
\newtheorem*{Satz*}{Theorem}
\newtheorem*{Rem*}{Remark}
\newtheorem*{Lem*}{Lemma}
\newtheorem{Satz}{Theorem}[section]
\newtheorem{Rem}{Remark}[section]
\newtheorem{Lem}{Lemma}[section]
\newtheorem{Prop}{Proposition}[section]
\newtheorem{Cor}{Corollary}[section]
\theoremstyle{definition}
\newtheorem{Def}[Satz]{Definition}

\parindent2ex

\newenvironment{rightcases}
  {\left.\begin{aligned}}
  {\end{aligned}\right\rbrace}

\begin{center}{\Large \bf A higher order TV-type variational problem related to the denoising and inpainting of images}
\end{center}
\begin{center}{\large \em Dedicated to Nicola Fusco on the occasion of his $60^{th}$ birthday}
\end{center}
\begin{center}
 M. Fuchs, J. M\"uller
\end{center}


\begin{abstract}
We give a comprehensive survey on a class of higher order variational problems which are motivated by applications in mathematical imaging. The overall aim of this note is to investigate if and in which manner results from the first author's previous work on variants of the  TV-regularization model (see e.g. \cite{BF1}, \cite{BF2}, \cite{BF3} and \cite{FT}) can be extended to functionals which involve higher derivatives. This seems to be not only of theoretical interest, but also relevant to applications  since higher order TV-denoising appears to maintain the advantages of the classical model as introduced by Rudin, Osher and Fatemi in \cite{ROF} while avoiding the unpleasant ''staircasing'' effect (see e.g. \cite{BKP} or \cite{LLT}). Our paper features results concerning generalized solutions in spaces of functions of higher order bounded variation, dual solutions as well as   partial regularity of minimizers. 
\end{abstract}

\begin{section}{Introduction}
Among the fundamental contributions of Nicola Fusco to the basic problems in the Calculus of Variations such as the question of (partial) regularity of weak solutions and the lower semicontinuity of variational integrals in various settings one will also find a number of important results addressing the popular field of free discontinuity problems as they for example occur in image analysis through the study of the Mumford-Shah functional. Without being complete we refer to the papers \cite{AFF}, \cite{FF} and advise the reader to consult chapters 6,7 and 8 of the monograph \cite{AFP} for more information on the subject and for further references.

In our note we concentrate on a particular aspect of image analysis, namely the simultaneous denoising and inpainting of images based on variants of the TV-regularization originally proposed by Rudin, Osher and Fatemi \cite{ROF} which in its simplest form consists in the problem of minimizing the functional
\begin{align}\label{1.1}
I[u]:=\intop_{\Omega}|\nabla u|+\frac{\lambda}{2}\intop_{\Omega-D}|u-f|^2\dx
\end{align}
among functions $u:\Omega\rightarrow\R^M$ with finite total variation (see e.g. \cite{Giu} or \cite{AFP}), where $\Omega$ denotes a bounded Lipschitz domain in $\R^n$, $n\geq 2$. In (\ref{1.1}) $D$ stands for a fixed subset of $\Omega$, and the given function $f:\Omega-D\rightarrow\R^M$ represents the observed data, which might be noisy. The idea is that on the set $D$ no observation is possible, so that by minimizing $I$ this missing observation can be retrieved from the measurement $f$ on $\Omega-D$ combined with a simultaneous denoising forced by the regularizing term $\intop_{\Omega}|\nabla u|$. The quantity $\intop_{\Omega-D}|u-f|^2\dx$ measures the quality of data fitting, and $\lambda$ denotes a positive parameter being under our disposal. In case $D=\emptyset$ the minimization of $I$ reduces to pure denoising, whereas for $D\neq\emptyset$ an inpainting procedure is incorporated. A clear interpretation of the above terminology can be given in the context of greyvalued images for which $n=2$, $M=1$ and $f(x)\in [0,1]$, $x\in\Omega-D$, measures the intensity of the grey level of the observed image. However, certain applications (see \cite{BC}, \cite{SS}) suggest to study even the higher dimensional case together with vectorial data.

From the analytical point of view the quantity $\intop_{\Omega}|\nabla u|$ for obvious reasons is rather unpleasant. At the same time - with regard to practical aspects - it seems to be desirable to keep the linear growth of the regularization term with respect to $|\nabla u|$, and as a compromise we replaced in the papers \cite{BF1}, \cite{BF2}, \cite{BF3} and \cite{FT} the functional $I$ from (\ref{1.1}) with the energy 
\begin{align}\label{1.2}
\tilde{I}[u]:=\intop_{\Omega}F(\nabla u)+\frac{\lambda}{2}\intop_{\Omega-D}|u-f|^2\dx
\end{align}
for a strictly convex density $F$ being of linear growth. A typical example (also of computational importance) is given by the formula
\begin{align}\label{1.3}
F(Z):=\Phi_\mu(|Z|),
\end{align}
where for $\mu>1$ we have set
\begin{align}\label{1.4}
\Phi_\mu(r):=\intop_0^r\intop_0^s(1+t)^{-\mu}\,dt\,ds,\;r\geq 0.
\end{align}
Note that we have the identities
\begin{align}\label{1.5}
\begin{cases}
 \Phi_\mu(r)=\frac{1}{\mu-1}r+\frac{1}{\mu-1}\frac{1}{\mu-2}(r+1)^{-\mu+2}-\frac{1}{\mu-1}\frac{1}{\mu-2},\;\mu\neq 2,\\
 \,\\
 \Phi_2(r)=r-\ln(1+r),\;r\geq 0,
\end{cases}
\end{align}
from which we deduce that $\Phi_\mu$ approximates the TV-density in the sense that
\begin{align}\label{1.6}
(\mu-1)\Phi_\mu(|Z|)\rightarrow|Z|\text{ as }\mu\rightarrow\infty.
\end{align}
It should be clear that the functional $\tilde{I}$ from (\ref{1.2}) has to be studied on the space  $BV(\Omega,\R^M)\cap L^2(\Omega-D,\R^M)$ (see e.g. \cite{Giu} or \cite{AFP}) with the interpretation of $\intop_{\Omega}F(\nabla u)$ as a convex function of a measure as introduced for example in \cite{DT}. Concerning the variational problem
\begin{align}\label{1.7}
\tilde{I}\rightarrow\min\text{ in }BV(\Omega,\R^M)\cap L^2(\Omega-D,\R^M)
\end{align}
we obtained the following results (compare e.g. \cite{BF1}, \cite{BF2}, \cite{BF3}, \cite{FT} and \cite{Ti}): 

\noindent\textit{Existence and uniqueness}:

Problem (\ref{1.7}) admits at least one solution being unique on $\Omega-D$. Moreover, the absolutely continuous part $\nabla^a u$ of $\nabla u$ with respect to Lebesgue's measure is uniquely determined. Any solution of problem (\ref{1.7}) occurs as a limit of an $\tilde{I}$-minimizing sequence from the space $W^{1,1}(\Omega,\R^M)\cap L^2(\Omega-D,\R^M)$. Here and in what follows we use the symbol $W^{m,p}(\Omega)$ for the standard Sobolev space (see \cite{Ad}) which is normed by 
\[
\|u\|_{m,p;\Omega}:=\sum_{k=0}^m\|\nabla^k u\|_{p;\Omega} 
\]
with $\nabla^0u:=u$ and $\|\,\cdot\,\|_{p;\Omega}$ denoting the $p$-norm. 

\noindent\textit{Duality}:

The problem dual to (\ref{1.7}) has a unique solution $\sigma\in L^\infty(\Omega,\R^{nM})$. The duality formula $\sigma=DF(\nabla^au)$ holds almost everywhere on $\Omega$.

\noindent\textit{Regularity}:

Any solution $u$ of problem (\ref{1.7}) is of class $C^1$ on an open subset of $\Omega$ whose complement is of Lebesgue measure zero. In the particular case of the density (\ref{1.3}) with $\Phi_\mu$ from (\ref{1.4}), (\ref{1.5}) (or for integrands of similar type) we have full interior regularity provided we assume $\mu<2$.

Apart from the first order TV-model and its extensions described above, higher order variants of the TV-model seem to be not only of theoretical interest as it is for example outlined in the recent paper \cite{BKP}. Roughly speaking, the functionals $I$ and $\tilde{I}$ from (\ref{1.1}) and (\ref{1.2}) are now replaced by the expressions
\begin{align*}
I^m[u]:=\intop_{\Omega}|\nabla^m u|+\frac{\lambda}{2}\intop_{\Omega-D}|u-f|^2\dx
\end{align*}
and
\begin{align*}
\tilde{I}^m[u]:=\intop_{\Omega}F(\nabla^mu)+\frac{\lambda}{2}\intop_{\Omega-D}|u-f|^2\dx,
\end{align*}
respectively, where $m\geq 2$ denotes some fixed integer. The functions $u:\Omega\rightarrow\R^M$ are taken from the space of {\em functions with $m$-th order bounded variation} defined by  
\begin{align*}
 BV^m(\Omega):=\big\{u\in W^{m-1,1}(\Omega)\,:\,\nabla^{m-1}u\in BV\big(\Omega,S^{m-1}(\R^n)\big)\big\},
\end{align*}
which means that a function $u$ is of $m$-th order bounded variation if and only if $u$ belongs to the Sobolev class $W^{m-1,1}(\Omega)$ with the additional property that the tensor of the $m$-th order generalized derivatives is a tensor valued Radon measure of finite total variation denoted by $|\nabla^m u|(\Omega)<\infty$. Here $S^m(\R^n)$ denotes the space of $m$-linear \textit{symmetric} maps $(\R^n)^m\rightarrow\R$. For notational simplicity we assume from now on that $M=1$.

The purpose of our note is the extension of the previously indicated results valid for the first order TV-model and its modifications formulated in (\ref{1.7}) to the higher order setting. To do so, we first fix our assumptions. In  what follows $\Omega$ always denotes a bounded Lipschitz domain in $\R^n$, $n\geq 2$. We consider either the case $D=\emptyset$ (''pure denoising'') or we assume that $D$ is a nonvoid subset of $\Omega$. In contrast to the first order case it might not be enough to require the set $D$ to be simply measurable: whenever the dimension $n$ of the ambient space $\R^n$ exceeds $2m$, the additional constraint $u\in L^2(\Omega-D)$ becomes nontrivial in the sense that the integrability does not follow from embedding theorems anyway. The same problem may occur when we consider more general data-fitting terms as considered in \cite{MT}. In both cases, our techniques rely on the following density result:

\begin{Satz}\label{Thm1.3}
\label{dens}
Let $\Omega\subset\R^n$ be a bounded Lipschitz domain and $D\subsetneq\Omega$ an open subset such that $\Omega-D$ has Lipschitz boundary as well. Then it holds for $1\leq p,q<\infty$:
\begin{enumerate}[(a)]
 \item  Any function $u\in W^{m,p}(\Omega)\cap L^q(\Omega-D)$ can be approximated by a sequence of smooth functions $(\varphi_k)_{k=1}^\infty\subset C^\infty(\overline{\Omega})$ such that
\begin{align*}
 \|u-\varphi_k\|_{m,p;\Omega}+\|u-\varphi_k\|_{q;\Omega-D}\rightarrow 0\text{ for }k\rightarrow\infty.
\end{align*}
 \item Any function $u\in BV^m(\Omega)\cap L^q(\Omega-D)$ can be approximated by a sequence of smooth functions $(\varphi_k)_{k=1}^\infty\subset C^\infty(\overline{\Omega})$ such that
\begin{align*}
  \|u-&\varphi_k\|_{m-1,1;\Omega}+\|u-\varphi_k\|_{q;\Omega-D}+\Bigl||\nabla^mu|(\Omega)-\intop_{\Omega}|\nabla^m\varphi_k|\dx\Bigr|\\
&+\Bigl|\sqrt{1+|\nabla^mu|^2}(\Omega)-\intop_{\Omega}\sqrt{1+|\nabla^m\varphi_k|^2}\dx\Bigr|\rightarrow 0\quad\text{ for }\quad k\rightarrow\infty.
\end{align*}
\end{enumerate}
\end{Satz}
Here, the term $\sqrt{1+|\nabla^m u|^2}(\Omega)$ has to be understood in the sense of convex functions of a measure; we once again refer to \cite{DT}. The reader should note that our hypotheses imposed on $\Omega$ and $D$ in particular imply that $\Omega-D$ is a set of positive Lebesgue measure, since the usual definition of ''Lipschitz boundary'' (cf., e.g. \cite{Ad}, 4.5, p. 66) requires the set to have nonempty interior.
\begin{Rem}
 In the earlier work \cite{Mü}, this theorem has been proved under stronger restrictions on the geometry of $\Omega$ and $D$. Nonetheless, quite recently we succeeded in generalizing this result towards the above assumptions. The proof is given in the subsequent section.
\end{Rem}
For $m\geq 2$ we consider a density $F:S^m(\R^n)\rightarrow[0,\infty)$ of class $C^1(S^m(\R^n))$ satisfying the following assumptions
\begin{align}
 &F \text{ is strictly convex and (w.l.o.g.) } F(0)=0,\label{1.8}\\ 
 &|DF(Z)|\leq\nu_1,\label{1.9}\\ 
 &F(Z)\geq\nu_2|Z|-\nu_3 \label{1.10}
\end{align}
with constants $\nu_1,\nu_2>0$, $\nu_3\in\R$, for all $Z\in S^m(\R^n)$. In accordance with (\ref{1.9}) and $F(0)=0$ we directly get
\begin{align*}
 F(Z)\leq \nu_1 |Z|
\end{align*}
for all $Z\in S^m(\R^n)$, and this inequality together with (\ref{1.10}) shows that $F$ is of linear growth in the sense that
\begin{align}\label{1.11}
 \nu_2|Z|-\nu_3\leq F(Z)\leq \nu_1|Z|,\;Z\in S^m(\R^n).
\end{align}
Note that our example from (\ref{1.3})-(\ref{1.5}) satisfies (\ref{1.8})-(\ref{1.11}). Suppose further that we are given a function $f$ at least of class $L^2(\Omega-D)$ and let $\lambda$ as before (see (\ref{1.1}) and (\ref{1.2})) denote a positive parameter. With these data we introduce the problem 
\begin{align}\label{1.12}
\begin{split}
 &J^m[u]:=\intop_{\Omega}F(\nabla^m u)\dx+\frac{\lambda}{2}\intop_{\Omega-D}|u-f|^2 \dx\rightarrow\min\\
&\hspace{6,5cm}\text{ in }W^{m,1}(\Omega)\cap L^2(\Omega-D), 
\end{split}
\end{align}
which due to the non-reflexivity of the Sobolev space $W^{m,1}(\Omega)$ admits in general no solution. As in the first order case we therefore pass to a suitable relaxed version: for functions $w$ from the space $BV^m(\Omega)\cap L^2(\Omega-D)$ we let
\begin{align}\label{1.13}
 K[w]:=\intop_\Omega F(\mu^a)\dx+\intop_{\Omega} F^\infty\left(\frac{\mu^s}{|\mu^s|}\right)d|\mu^s|+\frac{\lambda}{2}\intop_{\Omega-D}|w-f|^2\dx,
\end{align}
where we have abbreviated
\begin{align*}
 F^\infty:S^m(\R^n)\rightarrow[0,\infty),\;F^\infty(Z):=\lim_{t\rightarrow\infty}\frac{F(tZ)}{t},\; Z\in S^m(\R^n).
\end{align*}
$F^\infty$ is known as the recession function of the density $F$. In formula (\ref{1.13}) the symbol $\mu$ just denotes the measure $\nabla^mw$, for which we have the Lebesgue decomposition $\mu=\mu^a(\mathcal{L}^n\mres\Omega)+\mu^s$ with $\mu^a\in L^1(\Omega,S^m(\R^n))$ and $\mu^s\perp(\mathcal{L}^n\mres\Omega)$, $\mathcal{L}^n\mres\Omega$ denoting the restriction to $\Omega$ of Lebesgue's measure. Now we can state our first result concerning the relaxed version of problem (\ref{1.12}).

\begin{Satz}\label{Thm1.1}
 Let the  assumptions concerning $\Omega$ and $D$ stated in Theorem \ref{Thm1.3} hold and let $F$ satisfy (\ref{1.8})-(\ref{1.10}). Then it holds:
\begin{enumerate}[(a)]
 \item The problem
  \begin{align}\label{1.14}
   K\rightarrow\min\text{ in }BV^m(\Omega)\cap L^2(\Omega-D)
  \end{align}
   with $K$ from (\ref{1.13}) admits at least one solution $u$.
 \item Suppose that $u$ and $\tilde{u}$ are solutions of (\ref{1.14}). Then $u=\tilde{u}$ almost everywhere on $\Omega-D$ and $(\nabla^mu)^a=(\nabla^m\tilde{u})^a$ almost everywhere on $\Omega$.
 \item We have
   \begin{align*}
    \underset{W^{m,1}(\Omega)\cap L^2(\Omega-D)}{\inf}J^m=\underset{BV^m(\Omega)\cap L^2(\Omega-D)}{\inf}K
   \end{align*}
 where the functional $J^m$ is defined in (\ref{1.12}).
 \item Let $\mathcal{M}$ denote the set of all $W^{m-1,1}$-cluster points of $J^m$-minimizing sequences from the space $W^{m,1}(\Omega)\cap L^2(\Omega-D)$. Then $\mathcal{M}$ coincides with the set of all $K$-minimizers in $BV^m(\Omega)\cap L^2(\Omega-D)$.
 \item If $\mathcal{M}$ contains a function $u$ from the space $W^{m,1}(\Omega)\cap L^2(\Omega-D)$, then it holds $\mathcal{M}=\{u\}$.
\end{enumerate}
\end{Satz}

Let us pass to the variational problem being in duality to (\ref{1.12}) and (\ref{1.14}), respectively. To this purpose we introduce the Lagrangian by defining
\begin{align*}
l(w,\kappa):=\intop_{\Omega}\big[\kappa:\nabla^mw-F^*(\kappa)\big]\dx+\frac{\lambda}{2}\intop_{\Omega-D}|w-f|^2\dx,\\
w\in W^{m,1}(\Omega)\cap L^2(\Omega-D),\;\kappa\in L^\infty(\Omega,S^m(\R^n)),
\end{align*}
$F^*$ denoting the convex conjugate of $F$ and ''$:$'' the standard scalar product on $S^m(\R^n)$.  It holds
\begin{align}
J^m[w]=\underset{\kappa\in L^\infty(\Omega,S^m(\R^n))}{\sup} l(w,\kappa),\; w\in W^{m,1}(\Omega)\cap L^2(\Omega-D),
\end{align}
and we define the dual functional $R:L^\infty(\Omega,S^m(\R^n))\rightarrow[-\infty,\infty]$ through the formula
\begin{align*}
R[\kappa]:=\underset{w\in W^{m,1}(\Omega)\cap L^2(\Omega-D)}{\inf}l(w,\kappa).
\end{align*}

\begin{Satz}\label{Thm1.2}
Consider $\Omega$, $D$ and $F$  as in Theorem \ref{Thm1.1}. Then we have:
\begin{enumerate}[(a)]
\item The dual problem
 \begin{align*}
 R\rightarrow\max\text{ in }L^\infty(\Omega,S^m(\R^n))
\end{align*}
admits a unique solution $\sigma$.
\item $\sigma$ satisfies the duality relation $\sigma=DF\big((\nabla^mu)^a\big)$ almost everywhere on $\Omega$, where $u$ denotes any $K$-minimizer from the space $BV^m(\Omega)\cap L^2(\Omega-D)$.
\item The ''$\inf$-$\sup$ relation'' holds, i.e. we have
\begin{align*}
\underset{W^{m,1}(\Omega)\cap L^2(\Omega-D)}{\inf}J^m=\sup_{L^\infty(\Omega,S^m(\R^n))}R.
\end{align*}
\end{enumerate}
\end{Satz}

We finish our survey by adding a particular regularity result.

\begin{Satz}
\label{Thm1.4}
Let $\Omega$ denote a bounded Lipschitz domain in $\R^2$ and fix a measurable subset $D$ of $\Omega$ with $\mathcal{L}^2(\Omega-D)>0$. Consider an observed image $f$ in the space $L^\infty(\Omega-D)$. Moreover, suppose that $F$ is of class $C^2$ satisfying (\ref{1.8})-(\ref{1.10}) as well as the condition of $\mu$-ellipticity
\begin{align}
 \label{1.15} \nu_4\big(1+|Z|\big)^{-\mu}|X|^2\leq D^2F(Z)\big(X,X\big)\leq\nu_5\frac{1}{1+|Z|}|X|^2
\end{align}
 for some exponent $\mu>1$ and with positive constants $\nu_4,\nu_5$ for all $X,Z\in S^m(\R^2)$. Then, if 
\begin{align*}
 \text{either }&:\mu<2 \text{ together with } D=\emptyset\\
 \text{or \hspace{0.64cm}}&:\mu<\frac{3}{2} \text{ in case of general } D,
\end{align*}
the following statements hold:
\begin{enumerate}[(a)]
 \item Problem (\ref{1.12}) admits a unique solution $u$ in the space $W^{m,1}(\Omega)$.
 \item The function $u$ belongs to any class $W^{m,p}_{\mathrm{loc}}(\Omega)$, $p<\infty$.
 \item There is an open subset $\Omega_0$ of $\Omega$ with full $\mathcal{L}^2$-measure such that the minimizer $u$ is of class $C^{m,\alpha}(\Omega_0)$ for any $\alpha\in (0,1)$. In fact it holds $\mathcal{H}$-$\dim(\Omega-\Omega_0)=0$, i.e. $\mathcal{H}^\eps(\Omega-\Omega_0)=0$ for any $\eps>0$.
 \item If $\sigma$ denotes the solution of the dual problem from Theorem \ref{Thm1.2} (a), then the first weak partial derivatives of $\sigma$ exist in the space $L^2_{\mathrm{loc}}(\Omega)$.
\end{enumerate}
\end{Satz}

\begin{Rem}
 Condition (\ref{1.15}) holds for our example (\ref{1.3}) with $\Phi_\mu$ from (\ref{1.4}). It is easy to check that any density $F$ with (\ref{1.15}) for some $\mu>1$ is of linear growth in the sense of (\ref{1.11}).
\end{Rem}
\begin{Rem}\label{Rem1.2}
 Motivated by the results outlined in \cite{BF2}, \cite{FT} and \cite{Ti} we conjecture that actually $\mu<2$ is sufficient for obtaining the statements of Theorem \ref{Thm1.4} even in the case $D\neq \emptyset$. Moreover we think that $(c)$ can be improved to $u\in C^{m,\alpha}(\Omega)$, $0<\alpha<1$.
\end{Rem}
\begin{Rem}
 From the point of view of numerical applications it is desirable to replace the energies $J^m$ (from (\ref{1.12})) and $K$ (from (\ref{1.13})) by functionals of lower order involving appropriate coupling terms. This might also concern the aspect of regularity addressed in Remark \ref{Rem1.2}. The reader is referred to \cite{BFW}.
\end{Rem}

The rest of this article is organized as follows: in the next section we introduce the space $BV^m(\Omega)$ of functions of $m$-th order bounded variation along with two auxiliary results and prove our density result Theorem \ref{Thm1.3}. This is followed by a short section in which we sketch the proofs of both Theorem \ref{Thm1.1} and \ref{Thm1.2}. Since they only differ by the use of  Theorem \ref{Thm1.3} from the first order case studied in \cite{BF3} and \cite{FT}, we decided not to go into the details but rather to point out where  the density result finds application. The last part is entirely devoted to the proof of our regularity result Theorem \ref{Thm1.4}.

\end{section}

\begin{section}{The space $BV^m$, proof of Theorem \ref{Thm1.3}}\label{SecBV}
Several authors have conceptualized the notion of higher order bounded variation in varying ways, e.g. via the distributional Jacobian (cf. \cite{JS}). Despite that, the most natural generalization of this term appears to be saying  an $L^1$-function  is of $m$-th order bounded variation if it is $m-1$ times weakly  differentiable in $L^1$ and  its $m$-th order distributional gradient, i.e. the symmetric Tensor ($\partial_{i_1,...,i_m}u)_{i_1,...,i_m=1}^n$ is represented by a tensor valued finite Radon measure. As previously mentioned, we designate the space of all functions of $m$-th order bounded variation by
\begin{equation*}
 BV^m(\Omega):=\bigg\{u\in W^{m-1,1}(\Omega)\;:\;\nabla^{m-1}u\in BV\big(\Omega,S^m(\R^n)\big)\bigg\}
\end{equation*}
and note, that together with the norm
\[
 \|u\|_{BV^m(\Omega)}:=\|u\|_{m-1,1;\Omega}+|\nabla^mu|(\Omega)
\]
it becomes a Banach space. Further aspects of these (and even more general spaces) concerning also the approximation by smooth functions  have been outlined in \cite{DT} and \cite{De}.
\begin{Rem}
In the special case  $m=2$ the term  ''functions of bounded Hessian'' has  been established in \cite{De} by Demengel, and many authors prefer to write $\mbox{HB}(\Omega)$ (for \textit{''hessien borné''}) instead of $BV^2(\Omega)$, since. 
\end{Rem}
By the nature of its definition, the space $BV^m(\Omega)$ inherits the following compactness property which can be proved exactly as the corresponding first order result  (see \cite{AFP}, Theorem 3.23, p. 132):
\begin{Lem}[compactness in $BV^m$]\label{BVc}
Let $\Omega\subset\R^n$ be a bounded Lipschitz domain and $(u_k)_{k=1}^\infty\subset BV^m(\Omega)$ a sequence with
\[
 \|u_k\|_{BV^m(\Omega)}\leq M
\]
for some constant $M>0$. Then there is a subsequence $(u_{k_l})_{l=1}^\infty$ and  a function $u\in BV^m(\Omega)$ such that
\[
 \|u-u_{k_l}\|_{m-1,1;\Omega}\rightarrow 0\;\text{ for }l\rightarrow\infty\;\text{ and }\;\|u\|_{BV^m(\Omega)}\leq M.
\]
\end{Lem}

The following observation on equivalent norms in higher order spaces will be useful throughout our survey:

\begin{Lem}\label{Poinc} Let $\Omega\subset\R^n$ be a bounded Lipschitz domain, $m\in\N$, $1\leq p<\infty$ and let $D\subset\Omega$ be a measurable subset with $0\leq\mathcal{L}^n(D)<\mathcal{L}^n(\Omega)$.
\begin{enumerate}[(a)]
 \item There is a constant $C>0$, depending only on $\Omega$, $D$, $m,n$ and $p$ such that for all $u\in W^{m,p}(\Omega)$
\[
 \|u\|_{m,p;\Omega}\leq C\big(\|\nabla^{m}u\|_{p;\Omega}+\|u\|_{1,\Omega-D}\big).
\]
\item There is a constant $C>0$, depending only on $\Omega$, $D$, $m$ and $n$ such that for all $u\in BV^m(\Omega)$
\[
 \|u\|_{BV^m(\Omega)}\leq C\big(|\nabla^{m}u|(\Omega)+\|u\|_{1,\Omega-D}\big).
\]
\end{enumerate}
\end{Lem}
\begin{proof}
 \noindent (a) is a direct consequence of Theorem 1.1.15 in \cite{Ma} with the choice $\mathcal{F}(u):=\intop_{\Omega-D}|u|\dx$, since no non-trivial polynomial of degree $m-1$ vanishes on a set of positive Lebesgue measure.\\
 \noindent (b) follows from part (a) by approximation with smooth functions.
\end{proof}
For all $u,v\in BV^m(\Omega)$ we define a distance $d(.,.)$ by
\begin{align*}
d(u,v):=\|u-v\|_{m-1,1;\Omega}+\left|\int_{\Omega}|\nabla^m u|-\int_{\Omega}|\nabla^m v|\right|+\left|\int_{\Omega}f(\nabla^m u)-\int_{\Omega}f(\nabla^m v)\right|
\end{align*}
with  $f(X):=\sqrt{1+|X|^2}-1$ for $X\in S^m(\R^n)$.

We now give the \textit{\textbf{proof of Theorem \ref{Thm1.3}}} which states that smooth functions are dense in $BV^m(\Omega)\cap L^q(\Omega-D)$ with respect to this distance and the $L^q$-norm.

\noindent\textbf{\textit{ad (a)}}. We start with the following special case: Assume that $\Omega$ is a cuboid in $\R^n$,
\begin{align*}
 \Omega=(a_1,b_1)\times...\times(a_n,b_n)
\end{align*}
and $\Omega-D$ is given by
\begin{align*}
 \Omega-D=\big\{(x_1,...,x_n)\in\Omega\,\big|\,x_n<\phi(x_1,...,x_{n-1})\big\},
\end{align*}
where $\phi:(a_1,b_1)\times...\times (a_{n-1},b_{n-1})\rightarrow (a_n,b_n)$ is a Lipschitz continuous function with Lipschitz constant $L:=\text{Lip}(\phi)$. Consider now the sets 
\begin{align*}
 &\Omega_{-1}:=\emptyset,\\
 &\Omega_i:=\left\{x\in\Omega\,\big|\,\text{dist}(x,\partial{\Omega})>\frac{1}{i+1}\right\},\;i\in\N_0
\end{align*}
and consider the covering of $\Omega$ through the open sets $A_j:=\Omega_{j+1}-\overline{\Omega_{j-1}}$ for $j\in\N_{0}$. Let $(\eta_j)_{j=0}^\infty\subset C^\infty_0(\Omega)$ denote a partition of unity with respect to that covering, i.e. $\text{spt}\,\eta_j\Subset A_j$ and $\sum\limits_{j=0}^\infty \eta_j\equiv 1$. Let further $C$ denote the cone
\begin{align*}
 C:=\big\{(x_1,...,x_n)\in\R^n\,\big|\,x_n<-L|(x_1,...,x_{n-1},0)|\big\}
\end{align*}
and $\rho_\eps$ a symmetric mollifier supported in the ball $B_{\eps}(0)$. Note that for any $x\in \partial(\Omega-D)$ we have $(C+x)\cap\Omega\subset\Omega-D$.
For a given $\delta>0$, we will construct a function $\varphi_\delta\in C^\infty(\Omega)$ such that
\begin{align}\label{d0}
  \|u-\varphi_\delta\|_{m,p;\Omega}+\|u-\varphi_\delta\|_{q;\Omega-D}<\delta.
\end{align}
To this purpose, we consider $u_j:=\eta_ju$ and the shifted functions
\begin{align*}
u_j^{h_j}:=u_j(x_1,...,x_n-h_j) 
\end{align*}
for $h_j>0$ s.t. $\text{spt}(\eta_j)+h_j\Subset\Omega$. Since translations act continuously on $L^p(\R^n)$ (and hence, so they do on Sobolev spaces), we can choose a decreasing sequence of positive numbers $h_j$ such that $\text{spt}\,u_j+h_j\Subset A_j$ and
\begin{align}\label{d1}
  \|u_j-u_j^{h_j}\|_{m,p;\Omega}+\|u_j-u_j^{h_j}\|_{q;\Omega-D}<\delta 2^{-(j+2)}.
\end{align}
Further we can select a decreasing sequence of positive numbers $\eps_j$ which satisfy
\begin{align}
 &B_{\eps_j}(0)-(0,...,0,h_j)\Subset C \label{d2},\\
 &\big(\text{spt}(\eta_j)+h_j\big)^{\eps_j}\Subset A_{j+1}-A_{j-1} \label{lf}
\end{align}
and in addition
\begin{align}\label{d3}
 \|\rho_{\eps_j}\ast u_j^{h_j}-u_j^{h_j}\|_{m,p;\Omega}+\|\rho_{\eps_j}\ast u_j^{h_j}-u_j^{h_j}\|_{q;\Omega-D}<\delta 2^{-(j+2)}.
\end{align}
Note that due to (\ref{d2}) it actually holds $\rho_{\eps_j}\ast u_j^{h_j}\in L^q(\Omega-D)$ since $u_j^{h_j}$ is $q$-integrable on $\{x\in\Omega\,|\,x_n<\phi(x_1,...,x_{n-1})+h_j\}$.
By (\ref{d1}) and (\ref{d3}) we have
\begin{align}\label{d_4}
 \|u_j-\rho_{\eps_j}\ast u_j^{h_j}\|_{m,p;\Omega}+\|u_j-\rho_{\eps_j}\ast u_j^{h_j}\|_{q;\Omega-D}<\delta 2^{-(j+1)}.
\end{align}
Consequently, $\varphi_\delta:=\sum \limits_{j=0}^\infty \rho_{\eps_j}\ast u_j^{h_j}$ is a smooth function  that satisfies (\ref{d0}). Furthermore, by our construction we find that $u-\varphi_\delta\in \mathring{W}^{m,p}(\Omega).$

Now we consider the general case. Let $\delta>0$ be given. We can extend $u$ outside of $\Omega$ to a function in $W^{m,p}(\R^n)$ and therefore, w.l.o.g. assume that $\Omega-D$ is a compact subset of $\Omega$. We cover $\partial(\Omega-D)$ by a finite number of cuboids $Q_1,...,Q_N$ such that $(\Omega-D)\cap Q_i$ lies beneath the graph of a Lipschitz function, i.e. on each of the cuboids we are in the situation of our special case. Starting with $Q_1$, we can thus find a smooth function $\varphi_1\in C^\infty(Q_1)$ such that
\begin{align*}
 \|u-\varphi_1\|_{m,p;Q_1}+\|u-\varphi_1\|_{q;(\Omega-D)\cap Q_1}<\frac{\delta}{2N}
\end{align*}
and since $u-\varphi_1\in \mathring{W}^{m,p}(Q_1)$, the function $u_1$ defined through
\begin{align*}
 u_1(x):=
 \begin{cases}
  u(x),\;&x\in \Omega-Q_1,\\
  \varphi_1(x),\;&x\in Q_1
 \end{cases}
\end{align*}
is in $W^{m,p}(\Omega)\cap L^q(\Omega-D)$ and such that
\begin{align*}
 \|u-u_1\|_{m,p;\Omega}+\|u-u_1\|_{q;\Omega-D}<\frac{\delta}{2N}.
\end{align*}
Continuing this process on $Q_2$ with $u$ replaced by $u_1$ and so on, we finally end up with a function $u_N$ for which it holds
\begin{align*}
 \|u-u_N\|_{m,p;\Omega}+\|u-u_N\|_{q;\Omega-D}<\frac{\delta}{2}
\end{align*}
and which is smooth in an open neighbourhood $U$ of $\partial (\Omega-D)$. Choosing an open set $U_c\Subset U$ with $\partial(\Omega-D)\subset U_{c}$ and such that $\partial U_{c}$ is sufficiently regular, we get our desired approximating function by patching a suitable $C^{\infty}(\overline{\Omega}-\overline{U_{c}})$-approximation of $u|_{\overline{\Omega}-\overline{U_{c}}}$ with $u_N|_{U_c}$.

\noindent\textbf{\textit{ad (b)}}. We keep the notation from part (a). Again it will suffice to prove the claim in the special case of $\Omega$ being an $n$-dimensional cuboid and $\partial (\Omega-D)$ being the graph of a Lipschitz continuous function. For a given $\delta>0$ we choose a sequence $(h_j)_{j=0}^\infty$ of positive numbers such that $\text{spt}(\eta_j)+h_j\Subset A_j$ and the following conditions are satisfied:
\begin{align}
 &\|(\eta_ju)^{h_j}-\eta_j u\|_{m-1,1;\Omega}+\|(\eta_ju)^{h_j}-\eta_j u\|_{q;\Omega-D}<\delta2^{-(j+2)},\label{d1.6}\\
 &\big\|\big(\nabla^m(\eta_ju)-\eta_j\nabla^m u\big)^{h_j}-\big(\nabla^m(\eta_ju)-\eta_j\nabla^m u\big)\big\|_{1;\Omega}<\delta2^{-(j+2)}.\label{d1.7}
\end{align}
Note that $\nabla^m(\eta_ju)-\eta_j\nabla^m u\in L^1(\Omega)$ so that we can require $h_j$ to satisfy (\ref{d1.7}). Furthermore, since $\sum\limits_{j=0}^\infty\eta_j\equiv 1$ on $\Omega$ it holds $\sum\limits_{j=0}^\infty\big(\nabla^m(\eta_ju)-\eta_j\nabla^m u\big)\equiv 0$. Consider now 
\begin{align*}
 \udel:=\sum_{j=0}^\infty(\eta_ju)^{h_j}.
\end{align*}
We claim that $\udel$ approximates $u$ in $BV^m(\Omega)\cap L^q(\Omega-D)$ with respect to the metric $d(.,.)$. First, from (\ref{d1.6}) it follows 
\begin{align*}
 \|u-\udel\|_{m-1,1;\Omega}+\|u-\udel\|_{q;\Omega-D}<\delta/2.
\end{align*}
Let us compute the total variation of $\nabla^m\udel$. For a measure $\mu$ let $\mu^{h_j}$ denote the image measure under translation by $h_j$ in the $n$-th coordinate direction. 

By Proposition 3.18 in \cite{AFP} it holds $\nabla^m(\eta_ju)^{h_j}=\big(\nabla^m(\eta_ju)\big)^{h_j}$ and thus (recall $\sum\limits_{j=0}^\infty\eta_j\equiv 1$)
\begin{align*}
 &\nabla^m\udel=\sum_{j=0}^\infty\nabla^m(\eta_ju)^{h_j}=\sum_{j=0}^\infty\big(\nabla^m(\eta_ju)\big)^{h_j}\\
 &=\sum_{j=0}^\infty\big[(\eta_j\nabla^m u)^{h_j}+(\nabla^m(\eta_ju)-\eta_j\nabla^m u )^{h_j}(\mathcal{L}^n\mres\Omega)\big]\\
 &=\sum_{j=0}^\infty\big[(\eta_j\nabla^m u)^{h_j}+\big\{(\nabla^m(\eta_ju)-\eta_j\nabla^m u )^{h_j}-(\nabla^m(\eta_ju)-\eta_j\nabla^m u )\big\}(\mathcal{L}^n\mres\Omega)\big]
\end{align*}
so that
\begin{align*}
&\int_{\Omega}|\nabla^m\udel|\leq\sum_{j=0}^\infty\int_{\Omega}\big|(\eta_j\nabla^m u)^{h_j}\big|\\
&+\sum_{j=0}^\infty\int_{\Omega}\big|(\nabla^m(\eta_ju)-\eta_j\nabla^m u )^{h_j}-(\nabla^m(\eta_ju)-\eta_j\nabla^m u)\big|\dx\\
&\leq \sum_{j=0}^\infty\int_{\Omega}\big|(\eta_j\nabla^m u)^{h_j}\big|+\frac{\delta}{2}=\sum_{j=0}^\infty\int_{\Omega}\eta_j|\nabla^m u|+\frac{\delta}{2}=\int_{\Omega}|\nabla^mu| +\frac{\delta}{2}.
\end{align*}
Thus, for a sequence $\delta\downarrow 0$ we  find a corresponding sequence of functions $u_\delta$ for which $\limsup\limits_{\delta\rightarrow 0}|\nabla^m \udel|(\Omega)\leq |\nabla^mu|(\Omega)$ and  it follows from the lower semicontinuity of the total variation and $\udel\rightarrow u$ in $W^{m-1,1}(\Omega)$ that 
\begin{align*}
|\nabla^m \udel|(\Omega)\rightarrow |\nabla^mu|(\Omega).
\end{align*}
Moreover, we claim that it even holds 
\begin{align*}
(\nabla^m\udel)^a\rightarrow (\nabla^mu)^a\text{ in }L^1(\Omega,S^m(\R^n))
\end{align*} 
for $\delta\downarrow 0$ where, as before, for a tensor valued measure $\mu$ on $\Omega$ we denote by $\mu=\mu^a(\mathcal{L}^n\mres\Omega)+\mu^s$ its Lebesgue decomposition. To justify this, we first observe that on $\Omega_i$ we have 
\begin{align*}
 \nabla^m\udel|_{\Omega_i}=\Bigg(\sum_{j=0}^{i}\big(\nabla^m(\eta_ju)\big)^{h_j}\Bigg)\Bigg|_{\Omega_i}
\end{align*}
and since for two measures $\mu$ and $\nu$ it holds $(\mu+\nu)^a=\mu^a+\nu^a$ as well as $(\mu^{h_j})^a=(\mu^a)^{h_j}$, it follows
\begin{align*}
 \chi_{{}_{\Omega_i}}(\nabla^m\udel)^a=\chi_{{}_{\Omega_i}}\Bigg(\sum_{j=0}^{i}\big(\nabla^m(\eta_ju)^a\big)^{h_j}\Bigg)=\chi_{{}_{\Omega_i}}\Bigg(\sum_{j=0}^{\infty}\big(\nabla^m(\eta_ju)^a\big)^{h_j}\Bigg).
\end{align*}
But due to $\chi_{{}_{\Omega_i}}(\nabla^m\udel)^a\xrightarrow{i\rightarrow\infty}(\nabla^m\udel)^a$ in $L^1(\Omega,S^m(\R^n))$, it holds
\begin{align*}
 (\nabla^m\udel)^a=\sum_{j=0}^{\infty}\big(\nabla^m(\eta_ju)^a\big)^{h_j}.
\end{align*}
Since $\nabla^m(\eta_ju)^a\in L^1(\Omega,S^m(\R^n))$, we can choose $h_j$ small enough such that
\begin{align*}
 \|\big(\nabla^m(\eta_ju)^a\big)^{h_j}-\nabla^m(\eta_ju)^a\|_{1;\Omega}<\delta 2^{-(j+1)}
\end{align*}
and therefore $\|(\nabla^m\udel)^a-(\nabla^m u)^a\|_{1;\Omega}<\delta$. From $(\nabla^m\udel)^a\rightarrow(\nabla^m u)^a$ in $L^1(\Omega)$ it follows
\begin{align}
 \int_{\Omega}\sqrt{1+|(\nabla^m\udel)^a|^2}\dx\xrightarrow{\delta\downarrow 0}\int_{\Omega}\sqrt{1+|(\nabla^mu)^a|^2}\dx\label{d1.8}
\end{align}
and due to $|\nabla^m\udel|(\Omega)\rightarrow|\nabla^m u|(\Omega)$ and $|(\nabla^m\udel)^a|(\Omega)\rightarrow|(\nabla^m u)^a|(\Omega)$ it also holds 
\begin{align}
|(\nabla^m\udel)^s|(\Omega)\rightarrow|(\nabla^m u)^s|(\Omega).\label{d1.9}
\end{align}
Together with (\ref{d1.8}), this proves $\sqrt{1+|\nabla^m\udel|^2}(\Omega)\rightarrow \sqrt{1+|\nabla^m u|^2}(\Omega)$ for $\delta\downarrow 0$.

Now choose $h_j$ such that
\begin{align*}
 d(u,\udel)+\|u-\udel\|_{q;\Omega-D}<\delta/2.
\end{align*}
Note that $\eta_j\udel$ is $q$-integrable on $\{x\in\Omega\,|\,x_n<\phi(x_1,...,x_{n-1})+\tilde{h}_j\}$ with
\begin{align*} 
\tilde{h}_j:=\min\{h_{j-1},h_j,h_{j+1}\}. 
\end{align*}
 As outlined in \cite{DT} (after possibly adjusting $\Omega_0$), we can choose a sequence of positive numbers $(\eps_j)_{j=1}^\infty$ such that  the smooth functions 
\begin{align*}
\tilde{u}_\delta:=\sum\limits_{j=0}^\infty\rho_{\eps_j}\ast(\eta_j\udel)
\end{align*}
approximate $\udel$ in the sense that $d(\udel,\tilde{u}_\delta)<\delta/4$. Further, we can choose $\eps_j$ small enough to fulfill $B_{\eps_j}(0)-(0,...,0,\tilde{h}_j)\Subset C$ and $\|\udel-\tilde{u}_\delta\|_{q;\Omega}<\delta/4$. Altogether, $\tilde{u}_\delta$ approximates $u$ as required. The general case now follows as in part (a). \qed

\end{section}

\begin{section}{Generalized  and dual solutions}\label{Sec3}
In this section we are dealing with Theorems \ref{Thm1.1} and \ref{Thm1.2} concerning generalized solutions in $BV^m(\Omega)$ and dual solutions, respectively. Just for notational simplicity, we will confine ourselves to the case $m=2$ even though all results apply to arbitrary $m\in\N$ as well and abbreviate $J:=J^2$ in (\ref{1.12}). 
 As already mentioned in the introduction, both theorems will follow basically from the same arguments as in the first order case using the density result Theorem \ref{Thm1.3}. Thus, we are not going to give the details of the proofs but rather advise the reader to consult the work \cite{BF3} or \cite{FT} and restrict ourselves to an outline of how the density result is involved in each case. 

\noindent\textbf{\textit{ad Theorem \ref{Thm1.1}}}.  
First, using part (b) of Lemma \ref{Poinc}  it is clear that any $K$-minimizing sequence $(u_k)$ is uniformly bounded in $BV^2(\Omega)\cap L^2(\Omega-D)$ and hence, by the $BV^2$-compactness theorem and after passing to a suitable subsequence, there is a function $u\in BV^2(\Omega)$ such that $u_k\rightarrow u$ in $W^{1,1}(\Omega)$ and a.e. on $\Omega$. Furthermore, by an application of Fatou's lemma we see that in fact $u\in BV^2(\Omega)\cap L^2(\Omega-D)$.  Note that we have the following generalization of Lemma 3.1 in \cite{FT}, which  is a special case of the continuity theorem by Reshetnyak  (see \cite{Re} or \cite{AG} for a corrected version; compare also \cite{BS}, Theorem 2.4 and Remark 2.5):
\begin{Lem} For $w\in BV^2(\Omega)$ let
\begin{align*}
\tilde{K}[w]:=\intop_{\Omega}F\big((\nabla^2 w)^a\big)\dx+\intop_{\Omega}F^\infty\left(\frac{d(\nabla^2 w)^s}{d|(\nabla^2 w)^s|}\right)\,d|(\nabla^2 w)^s|.
\end{align*}
\begin{enumerate}[(a)]
\item Suppose $w_k,w\in BV^2(\Omega)$ are such that $w_k\rightarrow w$ in $W^{1,1}(\Omega)$ for $k\rightarrow\infty$. Then 
\begin{align}\label{3.1}
\tilde{K}[w]\leq \liminf_{k\rightarrow\infty}\tilde{K}[w_k].
\end{align}
\item If we know in addition 
\begin{align*}
\intop_{\Omega} \sqrt{1+|\nabla^2 w_k|^2}\rightarrow \intop_\Omega \sqrt{1+|\nabla^2 w|^2}\quad\text{for } k\rightarrow\infty,
\end{align*}
then $\tilde{K}[w]=\lim\limits_{k\rightarrow\infty}\tilde{K}[w_k]$.
\end{enumerate}
\end{Lem}
Hence, by (\ref{3.1})  we see that our limit function $u$ is in fact a minimizer of the generalized functional $K$ and from part (b) of the above lemma in combination with our density Theorem \ref{Thm1.3} (b) we conclude by the same reasoning as in \cite{FT}, that the infima of $J$ on $W^{2,1}(\Omega)\cap L^2(\Omega-D)$ and $K$ on $BV^2(\Omega)\cap L^2(\Omega-D)$ coincide. The remaining parts of Theorem \ref{Thm1.1} are proved following the lines of \cite{BF3}. \qed

\noindent\textbf{\textit{ad Theorem \ref{Thm1.2}}}. As in the first order case, we will obtain a maximizer of the dual functional $R$  from a sequence of solutions to a family of regularized problems.
For $1>\delta>0$ we consider
\begin{equation*}
 J_\delta[u]:=\intop_{\Omega}F_\delta(\nabla^2u)\dx+\frac{\lambda}{2}\intop_{\Omega-D}|u-f|^2\dx=\delta\intop_{\Omega}|\nabla^2u|^2\dx+I[u].
\end{equation*}
where
\begin{align*}
 F_\delta(P):=\delta|P|^2+F(P)\quad\text{for }P\in S^2(\R^n).
\end{align*}
Then we have 
\begin{Lem}\label{Lem41}
 \begin{enumerate}[ (a)]
  \item The problem $J_\delta\rightarrow\min$ in $W^{2,2}(\Omega)$ admits a unique solution $u_\delta$.
  \item We have  the uniform bound
 \begin{align}
  \sup_{\delta >0}\|u_\delta\|_{2,1;\Omega}<\infty.\label{4.1}
 \end{align}
 \item It holds  (not necessarily uniform in $\delta$) 
\[
u_\delta\in W^{3,2}_{\mathrm{loc}}(\Omega).
\]
\end{enumerate}
\end{Lem}

\noindent \textit{Proof.}$\;\,\mathbf{(a)}$: For fixed $\delta$ consider a $J_\delta$-minimizing sequence $v_k\in W^{2,2}(\Omega)$ for which we have
\begin{align}
 &\sup_{k} \|\nabla^2v_k\|_{2;\Omega}<\infty,\label{42}\\
 &\sup_k\|v_k\|_{2;\Omega-D}<\infty.\label{43}
\end{align}
Thus, by quoting Lemma \ref{Poinc} it is
\begin{align}
\sup_k\|v_k\|_{2,2;\Omega}<\infty.
\end{align}
After passing to a subsequence we may therefore assume $v_k\rightharpoondown:\overline{v}$ in $W^{2,2}(\Omega)$ for a function $\overline{v}$ from this space and standard results on lower semicontinuity imply the $J_\delta$-minimality of $\overline{v}$. If $\tilde{v}$ denotes a second $J_\delta$-minimizer, then the structure of $J_\delta$ clearly implies $\nabla^2\tilde{v}=\nabla^2\overline{v}$ on $\Omega$ together with $\tilde{v}=\overline{v}$ on $\Omega-D$, and since we assume $\mathcal{L}^2(\Omega-D)>0$ we get that $\tilde{v}=\overline{v}$ on $\Omega$. Thus $(a)$ of Lemma \ref{Lem41} is established.\\
\noindent $\mathbf{(b)}$: We essentially adopt the arguments used in part (a): in place of (\ref{42}) and (\ref{43}) it holds
\begin{align}
 &\sup_{\delta>0}\|\nabla^2u_\delta\|_{1;\Omega}<\infty,\label{410}\\
 &\sup_{\delta>0}\|u_\delta\|_{2;\Omega-D}<\infty,\\
 &\sup_{\delta>0}\delta\intop_{\Omega}|\hess\udel|^2\dx<\infty \label{deludel},
\end{align}
which is a consequence of $J_\delta[u_\delta]\leq J_\delta[0]=\frac{\lambda}{2}\intop_{\Omega-D}f^2\dx$. The $W^{2,1}(\Omega)$-version of Lemma \ref{Poinc} then yields 
\begin{align}
\sup_{\delta>0}\|\nabla u_\delta\|_{1;\Omega}<\infty \label{411}
\end{align}
and 
\begin{align}
\sup_{\delta>0}\|u_\delta\|_{1;\Omega}<\infty.\label{412}
\end{align}
Now our claim (\ref{4.1}) follows from (\ref{410}), (\ref{411}) and (\ref{412}). \\
\noindent$\mathbf{(c)}$: This is a standard application of the difference quotient technique (compare, e.g. \cite{BF0}, '\textit{Step 2}' on p. 353).

By the compactness property of $BV^2(\Omega)$ (cf. Lemma \ref{BVc}) and (\ref{4.1}) we thus infer that there is a function $u\in BV^2(\Omega)$ such that for a suitable sequence $\delta\downarrow 0$ it holds $u_\delta\rightarrow u$ in $W^{1,1}(\Omega)$. Furthermore, an application  of Fatou's Lemma gives $u\in BV^2(\Omega)\cap L^2(\Omega-D)$. In view of the duality relation stated in Theorem \ref{Thm1.2} (b) it is reasonable to consider $\sigma_\delta:=DF_\delta(\nabla^2 \udel)$ and $\tau_\delta:=DF(\nabla^2\udel)$ and to investigate their limiting behavior as $\delta\downarrow 0$. Indeed, by the same arguments as in \cite{BF3} we can verify that (for a subsequence) $\sigma_\delta$ has a weak limit $\sigma\in L^2(\Omega,S^2(\R^{n}))$ and $\tau_\delta$ has a weak-$*$-limit in $L^\infty(\Omega,S^2(\R^{n}))$. Moreover, it holds $\sigma=\tau$ a.e.

Our density Theorem \ref{Thm1.3} (a) comes into play when proving that $\tau$ in fact maximizes $R$. Revising the steps of the proof in \cite{BF3}, we find that the critical point is exactly the following:
after passing to the limit $\delta\downarrow 0$ in the Euler-Lagrange equation satisfied by the $J_\delta$-minimizer $u_\delta$
\begin{align}\label{ELd}
 \delta\intop_\Omega \nabla^2\udel:\nabla^2\varphi\dx+\intop_{\Omega}\tau_\delta:\nabla^2\varphi\dx+\lambda\intop_{\Omega-D}(\udel-f)\varphi\dx=0,
\end{align}
which holds for all $\varphi\in W^{2,2}(\Omega)$, we obtain (recall (\ref{deludel}) as well as the definition of $u$) 
\begin{align*}
 \intop_{\Omega}\tau:\nabla^2\varphi\dx+\lambda\intop_{\Omega-D}(u-f)\varphi\dx=0.
\end{align*}
We actually need the validity of this equation for all $\varphi$ in the domain of definition of $J$, i.e. for functions $\varphi\in W^{2,1}(\Omega)\cap L^2(\Omega-D)$. But this can easily be obtained by approximating $\varphi\in W^{2,1}(\Omega)\cap L^2(\Omega-D)$ with a sequence of smooth functions $(\varphi_k)\subset C^\infty(\overline{\Omega})$ in the sense of Theorem \ref{Thm1.3} (a). All the remaining statements of Theorem \ref{Thm1.2} can be verified without further difficulties following the lines of \cite{BF3}. \qed
\end{section}

\begin{section}{Partial regularity: proof of Theorem 1.4}
As before, for the sake of simplicity we shall confine ourselves to the case $m=2$ and remark, that throughout this section summation convention with respect to Greek indices is applied. In order to carry out the calculations during the proof of Theorem \ref{Thm1.4}, we rely on the $\delta$-approximation as introduced in the preceding section, cf. Lemma \ref{Lem41}. The next lemma is of crucial importance:
\begin{Lem}\label{Lem42}
 Let the assumptions of Theorem \ref{Thm1.4} hold, in particular we have (\ref{1.15}) with parameter $\mu\in (1,\tfrac{3}{2})$ in case $D\neq\emptyset$ and $\mu\in(1,2)$, if the case $D=\emptyset$ is considered. Then (with $\udel$ as in Lemma \ref{Lem41}) it holds
\begin{align}\label{4.14}
 \varphi_\delta:=\big(1+|\nabla^2u_\delta|\big)^{1-\frac{\mu}{2}}\in W^{1,2}_{\mathrm{loc}}(\Omega)\,\text{ uniformly in }\delta,
\end{align}
which means
\begin{align*}
 \sup_{\delta>0}\|\varphi_\delta\|_{1,2;\Omega^*}=c(\Omega^*)<\infty
\end{align*}
for any subdomain $\Omega^*\Subset\Omega$. Moreover, we have the uniform estimate
\begin{align}\label{est428}
 \sup_{\delta>0} \intop_{\Omega^*}D^2F_\delta(\hess\udel)\big(\palpha\hess\udel,\palpha\hess\udel\big)\dx\leq c(\Omega^*)<\infty.
\end{align}

\end{Lem}

Accepting Lemma \ref{Lem42} for the moment, we present the 

\noindent \textbf{\textit{proof of Theorem \ref{Thm1.4} part $\mathbf{(a)}$ and $\mathbf{(b)}$}}. As seen in section \ref{Sec3}, due to (\ref{4.1}) there is a function $u\in BV^2(\Omega)$ such that (for a suitable sequence $\delta\rightarrow 0$) we have $u_\delta\rightarrow u$ in $W^{1,1}(\Omega)$ (and even stronger convergences by the embedding theorem valid for $BV^2(\Omega)$). Quoting (\ref{4.14}) we find for any $p<\infty$
\begin{align}\label{4.15}
 \nabla^2u_\delta\in L^p_{\text{loc}}(\Omega,S^2(\R^2))\,\text{ uniformly in }\delta,
\end{align}
and by (\ref{4.15}) it is immediate that
\begin{align}\label{4.16}
 u\in W^{2,p}_{\text{loc}}(\Omega),\;p<\infty.
\end{align}
The estimate (\ref{est428}) together with the condition of $\mu$-ellipticity (\ref{1.15}) implies
\begin{align*}
 \sup_{\delta>0}\intop_{\Omega^*}\big(1+|\hess\udel|\big)^{-\mu}|\nabla^3\udel|^2\dx\leq c(\Omega^*)<\infty.
\end{align*}
For $s\in (1,2)$, in connection with (\ref{4.15}) this yields
\begin{align*}
\intop_{\Omega^*}|\nabla^3\udel|^s \dx&=\intop_{\Omega^*}\big(1+|\nabla^2\udel|\big)^{-\mu\frac{s}{2}}|\nabla^3\udel|^s\big(1+|\nabla^2\udel|\big)^{\mu\frac{s}{2}}\dx\\
&\leq\left(\intop_{\Omega^*}\big(1+|\nabla^2\udel|\big)^{-\mu}|\nabla^3\udel|^2 \dx\right)^{\frac{s}{2}}\left(\intop_{\Omega^*}\big(1+|\nabla^2\udel|\big)^\frac{\mu s}{2-s}\dx\right)^\frac{2-s}{2}\\
&\leq c(\Omega^*)<\infty,
\end{align*}
thus
\label{W3s}
\begin{align*}
\sup_{\delta>0}\|\udel\|_{3,s;\Omega^*}\leq c(s,\Omega^*)<\infty\quad\forall s<2
\end{align*}
and we therefore have strong convergence $u_\delta\rightarrow u$ in $W^{2,p}_\text{loc}(\Omega)$. All in all, we find that the following convergences hold true:
\label{conv}
\begin{align*}
 \begin{cases}
   \udel\rightharpoondown u\text{ in }W^{3,s}_{\text{loc}}(\Omega),\;s<2,\\
   \udel\rightarrow u \text{ in } W^{2,p}_\text{loc}(\Omega),\;p<\infty,\\
   \varphi_\delta\rightharpoondown\varphi \text{ in }W^{1,2}_{\text{loc}}(\Omega)
 \end{cases}
\end{align*}
with $\varphi=\big(1+|\nabla^2 u|\big)^{1-\frac{\mu}{2}}$, so that
\begin{align}\label{4.37}
\big(1+|\nabla^2 u|\big)^{1-\frac{\mu}{2}}\in W^{1,2}_\text{loc}(\Omega).
\end{align}

At the same time we have $W^{2,p}_{\text{loc}}(\Omega)\cap BV^2(\Omega)\subset W^{2,1}(\Omega)$ so that
\begin{align}\label{4.17}
 u\in W^{2,1}(\Omega)
\end{align}
on account of (\ref{4.16}). Thus $J[u]$ is well defined (recall (1.12) and $J=J^2$) and by lower semicontinuity it holds
\begin{align}\label{4.18}
 J[u]\leq\liminf_{\delta\rightarrow 0}J[u_\delta]
\end{align}
which is a consequence of $u_\delta\rightharpoondown u$ in $W^{2,p}_{\text{loc}}(\Omega)$ for any finite $p$. Quoting the $J_\delta$-minimality of $u_\delta$ we deduce from (\ref{4.18})
\begin{align*}
 J[u]\leq\liminf_{\delta\rightarrow 0}J_\delta[u_\delta]\leq J[v]
\end{align*}
for any $v\in W^{2,2}(\Omega)$. Therefore it holds 
\begin{align*}
 J[u]\leq J[w],\;w\in W^{2,1}(\Omega),
\end{align*}
since $\|w-v_k\|_{2,1;\Omega}\rightarrow 0$ for a suitable sequence $v_k\in W^{2,2}(\Omega)$. Altogether it is shown (recall (\ref{4.17})) that $u$ solves (\ref{1.12}). Assume that we have another minimizer $\widetilde{u}\neq u$ on a set of positive measure. From strict convexity we infer $\nabla^2\widetilde{u}=\nabla^2{u}$ a.e. on $\Omega$ as well as $\widetilde u= u$ a.e. on $\Omega-D$. Hence $u$ and $\widetilde{u}$ differ by a polynomial of degree at most $1$ which vanishes on $\Omega-D$. Since $\mathcal{L}^2(\Omega-D)>0$ it follows $\widetilde{u}=u$ a.e. This proves $(a)$ of Theorem \ref{Thm1.4}. For part $(b)$ we just refer to (\ref{4.16}).  \qed

We proceed with the

\noindent\textbf{\textit{proof of Lemma \ref{Lem42}}}. From Lemma \ref{Lem41} (c) and the $J_\delta$-minimality of $u_\delta$ it follows $(\alpha=1,2)$
\begin{align}\label{4.19}
 \intop_{\Omega}D^2F_\delta(\nabla^2u_\delta)(\partial_\alpha\nabla^2u_\delta,\nabla^2v)\dx=\lambda\intop_{\Omega-D}(u_\delta-f)\partial_\alpha v \dx
\end{align}
for any $v\in W^{2,2}(\Omega)$ such that $\text{spt }v$ is compactly contained in $\Omega$. Let us fix a disk $B_{2r}(x_0)$ contained in some arbitrary subregion $\Omega^*\Subset\Omega$. Consider a cut-off function $\eta$ such that 
\begin{align*}
 \begin{cases}
  0\leq\eta\leq1,\;\text{spt }\eta\subset B_{2r}(x_0),\\
  \eta=1\;\text{ on } B_r(x_0),\;|\nabla^l\eta|\leq cr^{-l},\;l=1,2
 \end{cases}
\end{align*}
and choose $v=\eta^6\partial_\alpha u_\delta$ in equation (\ref{4.19}). Abbreviating
\begin{align}
&T_1:=\intop_{\Omega}D^2F_\delta(\nabla^2u_\delta)(\partial_\alpha\nabla^2u_\delta,\partial_\alpha\nabla^2u_\delta)\eta^6 \dx \label{4.20}\\
&T_2:=\intop_{\Omega}D^2F_\delta(\nabla^2u_\delta)(\partial_\alpha\nabla^2u_\delta,2\nabla\eta^6\otimes\nabla\partial_\alpha u_\delta)\dx,\label{4.22}\\
&T_3:=\intop_{\Omega}D^2F_\delta(\nabla^2u_\delta)(\partial_\alpha\nabla^2u_\delta,\nabla^2\eta^6\partial_\alpha u_\delta) \dx\label{4.23}
\end{align}
with an obvious meaning of ``$\otimes$``, we obtain from (\ref{4.19})-(\ref{4.23})
\begin{align}\label{4.21}
 T_1+T_2+T_3=\lambda\intop_{\Omega-D}(u_\delta-f)\partial_\alpha(\eta^6\partial_\alpha u_\delta).
\end{align}
 Applying the Cauchy-Schwarz inequality to the bilinear form $D^2F_\delta(\nabla^2u_\delta)$ we get
\begin{align*}
 |T_2|\leq \intop_{\Omega}&\Big(D^2F_\delta(\nabla^2u_\delta)\big(\eta^3\palpha\hess\udel,\eta^3\palpha\hess\udel\big)\Big)^{\frac{1}{2}}\\
 &\cdot \Big(D^2F_\delta(\nabla^2u_\delta)\big(12\eta^2\nabla\eta\otimes\palpha\nabla\udel, 12\eta^2\nabla \eta\otimes\palpha\nabla\udel\big)\Big)^{\frac{1}{2}}\dx,
\end{align*}
hence, by Young's inequality
\begin{align*}
 |T_2|&\leq\eps T_1+c(\eps)\intop_{\Omega}\big|D^2F_\delta(\hess\udel)\big|\eta^4|\nabla\eta|^2|\hess\udel|^2\dx\\
 &=:\eps T_1+c(\eps)T_4.
\end{align*}
In $T_4$ we make use of inequality (\ref{1.15}), the bound stated in (\ref{4.1}) and the fact that
\begin{align*}
 \sup_{\delta>0}\delta\intop_{\Omega}|\hess\udel|^2\dx<\infty
\end{align*}
(see (\ref{deludel})) with the result $T_4\leq c(\eta)$, so that
\begin{align}\label{4.24}
 |T_2|\leq\eps T_1+c(\eps,\eta).
\end{align}
In a similar way we obtain
\begin{align*}
|T_3|\leq\eps T_1+c(\eps,\eta)\intop_{\Omega}|D^2F_\delta(\hess\udel)||\nabla\udel|^2\dx
\end{align*}
and since for example
\begin{align*}
\sup_{\delta>0}\|\nabla u_\delta\|_{2;\Omega}<\infty
\end{align*}
on account of (\ref{4.1}) and Sobolev's embedding theorem, we find
\begin{align}\label{4.25}
|T_3|\leq \eps T_1+c(\eps,\eta).
\end{align}
Putting together (\ref{4.21})-(\ref{4.25}) and choosing $\eps=\frac{1}{4}$ it follows
\begin{align}\label{4.26}
\begin{split}
&\intop_{B_{2r}(x_0)}D^2F_\delta(\hess\udel)(\palpha\hess\udel,\palpha\hess\udel)\eta^6\dx\\
&\leq c(\eta)+2\lambda\underset{\mbox{$=:T_5$}}{\underbrace{\intop_{\Omega-D}(u_\delta-f)\palpha(\eta^6\palpha\udel)\dx}}.
\end{split}
\end{align}

\noindent\textbf{\textit{Case 1}}: $D=\emptyset$

It holds
\begin{align*}
T_5&=\intop_{B_{2r}(x_0)}\udel\palpha(\eta^6\palpha\udel)\dx-\intop_{B_{2r}(x_0)}f\palpha(\eta^6\palpha\udel)\dx\\
&=-\intop_{B_{2r}(x_0)}\eta^6|\nabla\udel|^2\dx-\intop_{B_{2r}(x_0)}f\palpha(\eta^6\palpha\udel)\dx,
\end{align*}
and by the boundedness of $f$ we get from (\ref{4.26})
\begin{align}\label{4.27}
\begin{split}
 \intop_{B_{2r}(x_0)}D^2F_\delta(\nabla^2\udel)(\palpha\nabla^2\udel,\palpha\nabla^2\udel)\eta^6 \dx+2\lambda\intb\eta^6|\nabla\udel|^2\dx\\
\leq c(\eta)+\|f\|_{\infty;\Omega}\left\{\intb|\nabla\eta^6||\nabla\udel|\dx+\intb\eta^6|\hess\udel|\dx\right\}.
\end{split}
\end{align}
Clearly all terms on the right-hand side of (\ref{4.27}) are bounded independent of $\delta$, hence (\ref{4.27}) implies
\begin{align}\label{4.28}
\sup_{\delta>0}\intop_{\Omega^*}D^2F_\delta(\hess\udel)(\palpha\nabla^2\udel,\palpha\nabla^2\udel)\dx<\infty.
\end{align}
Recalling the definition of $\varphi_\delta$ stated in (\ref{4.14}) it is immediate that
\begin{align}\label{4.29}
\sup_{\delta>0}\|\varphi_\delta\|_{2;\Omega}<\infty.
\end{align}
At the same time we obtain from (\ref{1.15}) and (\ref{4.28})
\begin{align*}
\sup_{\delta>0}\|\nabla\varphi_\delta\|_{2;\Omega^*}<\infty,
\end{align*}
which together with (\ref{4.29}) proves our claims (\ref{4.14}) and (\ref{est428}). We wish to note that the requirement $\mu<2$ just enters through the fact that in the definition of $\varphi_\delta$ the exponent $1-\frac{\mu}{2}$ must be positive.

\noindent\textbf{\textit{Case 2}}: $D\neq\emptyset$

Again we look at the relevant item $T_5$ in (\ref{4.26}) and recall that the part $\intop_{\Omega-D}f\palpha(\eta^6\palpha\udel)\dx$ is uncritical since by (\ref{4.1}) its absolute value can be bounded by a constant $c(\eta)$. So let
\begin{align*}
 \widetilde{T_5}:=\intop_{\Omega-D}\udel\palpha(\eta^6\palpha\udel)\dx
\end{align*}
and observe
\begin{align}\label{4.30}
|\widetilde{T_5}|\leq \into|\udel||\nabla\eta^6||\nabla\udel|\dx+\into|\udel|\eta^6|\hess\udel|\dx=:T_6+T_7.
\end{align}
From (\ref{4.1}) and Sobolev's embedding theorem it is immediate that 
\begin{align}\label{4.31}
T_6\leq c(\eta).
\end{align}
To the quantity $T_7$ we apply Young's inequality
\begin{align}\label{4.32}
\begin{split}
\begin{rightcases}
T_7&\leq c\left(\into|\udel|^{\frac{\mu}{\mu-1}}\eta^6\dx+\into\eta^6|\hess\udel|^\mu \dx\right)\\
&\underset{(\ref{4.1})}{\leq}c\left(1+\into\eta^6|\hess\udel|^\mu \dx\right) \\
&\leq c\left(1+\into (\eta^3\psi_\delta)^2 \dx\right)
\end{rightcases}
\end{split}
\end{align}
with $\psi_\delta:=\big(1+|\hess\udel|\big)^{\frac{\mu}{2}}$.
Sobolev's inequality yields 
\begin{align}\label{4.33}
\begin{split}
\into(\eta^3\psi_\delta)^2 \dx\leq c\left(\into|\nabla (\eta^3\psi_\delta)|\dx\right)^2\leq c(\eta)+c \left(\into \eta^3|\nabla\psi_\delta|\dx\right)^2,
\end{split}
\end{align}
where in the last step we used (\ref{4.1}) together with $\mu<2$. Observing (recall (\ref{4.14})) the relation
\begin{align*}
\psi_\delta=\varphi_\delta^{\frac{\mu}{2-\mu}}
\end{align*}
we obtain
\begin{align}\label{4.34}
\begin{split}
\intb\eta^3|\nabla \psi_\delta|\dx\leq c\intb\eta^3|\nabla\varphi_\delta|\varphi_\delta^{\frac{\mu}{2-\mu}-1}\dx\\
\leq c\left(\intb\eta^6|\nabla\varphi_\delta|^2 \dx\right)^\frac{1}{2}\left(\intb\varphi_\delta^\frac{4\mu-4}{2-\mu}\dx\right)^\frac{1}{2}.
\end{split}
\end{align}
Inserting (\ref{4.34}) into (\ref{4.33}) and going back to (\ref{4.32}) we obtain
\begin{align}\label{4.35}
T_7\leq c(\eta)+c\intb\varphi_\delta^{\frac{4\mu-4}{2-\mu}}\dx\intb\eta^6|\nabla\varphi_\delta|^2\dx.
\end{align}
Recalling (\ref{4.30})  and (\ref{4.31}), we deduce from (\ref{4.35})
\begin{align}\label{4.36}
\big|\widetilde{T_5}\big|\leq c(\eta)+c\intb\varphi_\delta^\frac{4\mu-4}{2-\mu}\dx\intb\eta^6|\nabla\varphi_\delta|^2\dx.
\end{align}
With (\ref{4.36}) at hand we return to (\ref{4.26}) observing first
\begin{align*}
|\nabla\varphi_\delta|^2\leq c\, D^2F_\delta(\hess\udel)(\palpha\hess\udel,\palpha\hess\udel).
\end{align*}
Second we note
\begin{align*}
\varphi_\delta^\frac{4\mu-4}{2-\mu}= \left(1+|\hess\udel|\right)^{\frac{2-\mu}{2}\frac{4\mu-4}{2-\mu}}\sim |\hess\udel|^{2\mu-2}
\end{align*}
with exponent $2\mu-2<1$ on account of our assumption $\mu<\frac{3}{2}$ in Case 2. Thus, by (\ref{4.1}) and by Hölder's inequality we obtain
\begin{align*}
\intb\varphi_\delta^\frac{4\mu-4}{2-\mu}\dx\leq c\mathcal{L}^2\left(B_{2r}(x_0)\right)^s
\end{align*}
for some positive exponent $s$ and  (\ref{4.26}) in combination with our previous results yields
\begin{align*}
\left(1-c\mathcal{L}^2\left(B_{2r}(x_0)\right)^s\right)\intb \eta^6D^2F_\delta(\hess\udel)(\palpha\hess\udel,\palpha\hess\udel)\leq c(\eta).
\end{align*}
From this inequality we deduce: if we restrict ourselves to radii $r\leq r_0$ for some $r_0>0$ independent of $\delta$, then it holds $\intb|\nabla\varphi_\delta|^2\dx\leq c(r)<\infty$, thus (\ref{4.14}) and (\ref{est428}) are established in Case 2. \qed

We now turn to the

\noindent \textbf{\textit{proof of Theorem \ref{Thm1.4} part $\mathbf{(c)}$}}.  The basis for our further proof is the well known blow-up technique whose idea together with appropriate references is explained in the monograph \cite{Gia}. We also suggest to consult the paper \cite{EG}. Our arguments follow these ideas and their higher order version given in \cite{AF}, section 3, where now some adjustments become necessary due to the presence of the error term $\intop_{\Omega-D}(u-f)^2\dx$ in our situation.

For disks $B_\rho(x)\Subset\Omega$ we define the  excess function  by
\begin{align*}
E(x,\rho):=\mint \limits_{B_\rho(x)}\left|\hess u(y)-(\hess u)_{x,\rho}\right|^2 dy,
\end{align*}
where $u\in W^{2,1}(\Omega)$ is the minimizer from part (a). Note, that due to part (b) of Theorem \ref{Thm1.4} the excess is well defined. By $(\hess u)_{x,\rho}:=\txtmint_{B_{\rho}(x)}\hess u(y)\,dy$ we denote the mean value of $\nabla^2 u$ on the disk $B_\rho(x)$. The essential step is to show the following excess-decay lemma:
\begin{Lem}[Blow-up Lemma]\label{Bl}
Given $L>0$, define the constant $C^*(L)$ according to (\ref{4.48}) below and set $C_*=C_*(L):=2C^*(L)$. Then, for any $\tau\in (0,\frac{1}{2})$ there is an $\eps=\eps(L,\tau)$ such that whenever
\begin{align}\label{4.38}
|(\hess u)_{x,r}|\leq L,\; E(x,r)+r\leq \eps,
\end{align}
then also
\begin{align}\label{4.39}
E(x,\tau r)\leq \tau^2 C_*(L)(E(x,r)+r)
\end{align}
for disks $B_r(x)\Subset\Omega$.
\end{Lem}
\begin{Rem}
\begin{enumerate}[(a)]\label{rembl}
\item Due to Lebesgue's differentiation theorem, condition (\ref{4.38}) is valid for $\mathcal{L}^2$-almost all points $x\in \Omega$, i.e. the set
\begin{align*}
 \Omega_0:=\left\{x\in\Omega\,:\,\limsup_{r\rightarrow 0}|(\hess u)_{x,r}|<\infty\right\}\cap\left\{x\in\Omega\,:\,\liminf_{r\rightarrow 0}E(x,r)=0\right\}
\end{align*}
 has full Lebesgue measure.
\item That in fact $\nabla^2 u$ is Hölder continuous on $\Omega_0$ and that $\Omega_0$ is an open subset of $\Omega$ follows from Lemma \ref{Bl} in a standard way, as e.g. outlined in detail on p. 95 ff. in the monograph \cite{Gia}:
by iteration, inequality (\ref{4.39}) yields ($0<\alpha<1$)
\begin{align*}
 E(x,\tau r)\leq \tau^{2\alpha}(E(x,r)+r),
\end{align*}
where $\tau$ is such that $C_*(L)\tau^{2-2\alpha}=1$. This implies
\begin{align*}
 E(x,\rho)\leq c \left(\frac{\rho}{r}\right)^{2\alpha}(E(x,r)+r)
\end{align*}
for all $\rho\leq r$ and from this inequality together with Morrey's integral characterization of Hölder continuity (cf. \cite{Gia}, chapter III,  Theorem 1.3) we get our claim.
\end{enumerate}
\end{Rem}

\noindent \textbf{\textit{Indirect proof of the Blow-up Lemma}}.
Fix $L>0$. If the statement of the lemma is false, then there is a $\tau\in(0,\frac{1}{2})$ and a sequence $B_{r_k}(x_k)\Subset\Omega$ of disks with
\begin{align}\label{4.40}
\left|(\hess u)_{x_k,r_k}\right|\leq L,\;E(x_k,r_k)+r_k=:\lambda^2_k\rightarrow 0,
\end{align}
but at the same time
\begin{align}\label{4.41}
E(x_k,\tau r_k)>\tau^2C_*(L)\lambda^2_k.
\end{align}
We rescale the function $u$ by setting
\begin{align*}
&a_k:=(u)_{x_k,r_k},\; A_k:=(\nabla u)_{x_k,r_k},\;\Theta_k=(\hess u)_{x_k,r_k},\\
&u_k(z):=\frac{1}{\lambda_kr_k^2}\left[u(x_k+r_kz)-r_kA_k\cdot z-a_k-\frac{r_k^2}{2}\Theta_k(z,z)+\frac{r_k^2}{2}\mint \limits_{B_1(0)}\Theta_k(y,y)dy\right],
\end{align*}
\label{defuk} where $z\in B_1(0)$.
These scalings are chosen in such a way that $(u_k)_{0,1}=0$, $(\nabla u_k)_{0,1}=0$, $(\hess u_k)_{0,1}=0$ and we further have
\begin{align*}
&\nabla u_k(z)=\frac{1}{\lambda_k r_k}\left[\nabla u(x_k+r_kz)-A_k-\frac{1}{2}r_k\nabla(\Theta_k^{\alpha\beta}z_\alpha z_\beta)\right],\\
&\nabla^2 u_k(z)=\frac{1}{\lambda_k}\left[\nabla^2u(x_k+r_kz)-\Theta_k\right]
\end{align*}
as well as
\begin{align}\label{4.42}
\mint\limits_{B_1(0)}|\nabla^2 u_k(z)|^2\dz=\lambda_k^{-2}E(x_k,r_k)\underset{(\ref{4.40})}{\leq}1.
\end{align}
The vanishing of the averages along with (\ref{4.42}) and Poincaré's inequality yields (after passing to a suitable subsequence)
\begin{align}\label{4.43}
u_k\rightharpoondown:\hat{u}\;\text{in }W^{2,2}(B_1(0))
\end{align}
and consequently
\begin{align}\label{4.44}
\lambda_k\nabla^2u_k\rightarrow 0\text{ in }L^2(B_1(0))\text{ and a.e.}
\end{align}
According to (\ref{4.40}) we further have
\begin{align}\label{4.45}
\Theta_k\rightarrow:\Theta
\end{align}
for a $2\times 2$-matrix $\Theta$. We claim that the function $\hat{u}$ fulfills the following constant coefficient elliptic system (implying the smoothness of $\hat{u}$):
\begin{align}\label{4.46}
\intop_{B_1(0)}D^2F(\Theta)(\nabla^2\hat{u},\nabla^2\psi)\dz=0\;\forall\psi\in C^\infty_0(B_1(0)).
\end{align}
 Fix $\psi$ and set $\varphi (x):=\psi\left(\frac{x-x_k}{r_k}\right)$, $x\in B_{r_k}(x_k)$. By the minimality of $u$, it holds
\begin{align*}
0=\underset{\mbox{$=:S_1$}}{\underbrace{\intop_{B_{r_k}(x_k)}DF(\nabla^2u):\nabla^2\varphi \dx}}+\lambda \underset{\mbox{$=:S_2$}}{\underbrace{\intop_{B_{r_k}(x_k)-D}(u-f)\varphi \dx}}
\end{align*}
 On any open subset $\Omega^*\Subset\Omega$, both $u$ and $f$ are bounded (recall $u\in W^{3,s}_{\text{loc}}(\Omega)$, $s<2$). Thus we can estimate $S_2$ by
\begin{align*}
|S_2|\leq c\intop_{B_{r_k}(x_k)}|\varphi|\dx=c\intop_{B_{r_k}(x_k)}\left|\psi\left(\frac{x-x_k}{r_k}\right)\right|\dx=cr_k^2\intop_{B_1(0)}|\psi(z)|\dz\\
\leq C(\psi)r_k^2.
\end{align*}
After  transformation the integral $S_1$ becomes
\begin{align*}
S_1&=\intop_{B_1(0)}DF(\Theta_k+\lambda_k\hess u_k):\hess \psi \dz\\
       &=\intop_{B_1(0)}\intop_0^1\frac{d}{ds}DF(\Theta_k+s\lambda_k\hess u_k):\hess\psi \,ds\,dz\\
       &=\intop_{B_1(0)}\intop_0^1D^2F(\Theta_k+s\lambda_k\hess u_k)(\hess u_k,\hess \psi)\lambda_k \,ds\,dz
\end{align*}
and together with our estimate for $S_2$ this yields
\begin{align}\label{4.47}
\Bigg|\intop_{B_1(0)}\intop_0^1D^2F(\Theta_k+s\lambda_k\hess u_k)(\hess u_k,\hess\psi)\,ds\,dz\Bigg|\leq C(\psi)\lambda_k^{-1}r_k^2.
\end{align}
Because of \label{null}(\ref{4.40}), $\frac{r_k}{\lambda_k}\leq\lambda_k\rightarrow 0$ for $k\rightarrow\infty$ and thus (note that $r_k$ is bounded) also $\frac{r_k^2}{\lambda_k}\rightarrow 0$. 
Now we turn to the left-hand side of (\ref{4.47}). Let $\delta>0$ be given. By (\ref{4.44}) and Egorov's Theorem there is a set $S\subset B_1(0)$ with $\mathcal{L}^2(B_1(0)-S)<\delta$ and $\lambda_k\hess u_k\rightrightarrows 0$ a.e. on $S$. With (\ref{4.43}) and (\ref{4.44}) it follows:
\begin{align*}
\intop_{S}\intop_0^1D^2F(\Theta_k+\lambda_k s\hess u_k)(\hess u_k,\hess \psi)\,ds\,dz\rightarrow\intop_{S}D^2F(\Theta)(\hess \hat{u},\hess\psi)\dz.
\end{align*}
At the same time, due to the boundedness of $D^2F$ and by Hölder's inequality we find 
\begin{align*}
&\Bigg|\intop_{B_1(0)-S}\intop_0^1 D^2F(\Theta_k+\lambda_k s\hess u_k)(\hess u_k,\hess\psi)\, ds\,dz\Bigg|\\
&\leq c\|\hess u_k\|_{2;B_1(0)}\|\hess \psi\|_{2;B_1(0)-S}\\
&\underset{(\ref{4.43})}{\leq}C\|\hess\psi\|_{2;B_1(0)-S}\leq C(\psi)\delta^{\frac{1}{2}},
\end{align*}
and since $\delta$ can be chosen arbitrarily small, this proves
\begin{align*}
\intop_{B_1(0)}\intop_0^1D^2F(\Theta_k+s\lambda_k\hess u_k)(\hess u_k,\hess\psi)\,ds\,dz\rightarrow\intop_{B_1(0)}D^2F(\Theta)(\hess\hat{u},\hess\psi)\dz
\end{align*}
and (\ref{4.46}) follows. We can therefore rely on the results in \cite{GM} and \cite{Kr} on  higher order elliptic systems (see also the comments subsequent to (3.10) in \cite{AF}) and find  a constant $C^*(L)$ such that
\begin{align}\label{4.48}
\mint\limits_{B_\tau(0)}|\hess\hat{u}-(\hess\hat{u})_{0,\tau}|^2\dz\leq C^*(L)\tau^2,
\end{align}
which together with the definition of $C_*(L)$ contradicts (\ref{4.41}) as soon as (\ref{4.43}) is improved towards
\begin{align}\label{4.49}
\hess u_k\rightarrow\hess \hat{u}\text{ in }L^2_\text{loc}(B_1(0)).
\end{align}
In fact, after scaling (\ref{4.41}) reads
\begin{align*}
\mint\limits_{B_\tau(0)}|\hess u_k-(\hess u_k)_{0,\tau}|^2 \dz=\lambda_k^{-2}E(x_k,\tau r_k)>\tau^2 C_*(L),
\end{align*}
and hence, along with (\ref{4.49})
\begin{align*}
\mint \limits_{B_\tau(0)}|\hess\hat{u}-(\hess\hat{u})_{0,\tau}|^2\dz\geq\tau^2 C_*(L).
\end{align*}
Therefore, in order to complete the proof of the Blow-up Lemma we need to verify (\ref{4.49}). To do this, we proceed just like in \cite{AF} and notice that we have (cf. (3.14) therein)
\begin{align}\label{4.50}
\lim_{k\rightarrow\infty}\intop_{B_\rho(0)}\big(1+|\Theta_k|+\lambda_k|\hess\hat{u}|+\lambda_k|\nabla^2 w_k|\big)^{-\mu}|\hess w_k|^2 \dz=0
\end{align}
where $\rho\in (0,1)$ and $w_k:=u_k-\hat{u}$. Since the derivation of (\ref{4.50}) is somewhat lengthy, we postpone its proof to the end and continue to establish (\ref{4.49}). 
Fix $\rho\in (0,1)$ and choose $M\geq 1$: it holds
\begin{align*}
 &\intop_{B_\rho(0)}|\hess w_k|^2\dz=\intop_{B_\rho(0)\cap\big[\lambda_k|\hess u_k|\leq M\big]}|\hess w_k|^2 \dz+\eps_k,\\
 &\eps_k:=\intop_{B_\rho(0)\cap\big[\lambda_k|\hess u_k|> M\big]}|\hess w_k|^2 \dz.
\end{align*}
Due to $\nabla^2\hat{u}\in L^\infty_{\text{loc}}(B_1(0))$ and the boundedness of the sequence $\Theta_k$ equation (\ref{4.50}) implies
\begin{align}\label{4.51}
 \lim_{k\rightarrow\infty}\intop_{B_\rho(0)\cap\big[\lambda_k|\nabla^2u_k|\leq M\big]}|\hess w_k|^2 \dz=0.
\end{align}
Let
\begin{align*}
\varphi_k:=\lambda_k^{-1}\left\{\big(1+|\Theta_k+\lambda_k\hess u_k|\big)^{1-\frac{\mu}{2}}-\big(1+|\Theta_k|\big)^{1-\frac{\mu}{2}}\right\},
\end{align*}
i.e. we consider the scaled version of $\varphi$ from (\ref{4.37}).
We claim the validity of 
\begin{align}\label{4.52}
\sup_k\intop_{B_\rho(0)}|\nabla\varphi_k|^2 \dz\leq c(\rho)<\infty.
\end{align}
Accepting this inequality for the moment, we further observe
\begin{align*}
|\varphi_k|=\lambda_k^{-1}\left|\intop_0^1\frac{d}{ds}\big(1+|\Theta_k+s\lambda_k\hess u_k|\big)^{1-\frac{\mu}{2}} \,ds\right|\leq c|\hess u_k|,
\end{align*}
so that (\ref{4.42}) implies the $L^2(B_1(0))$-boundedness of the $\varphi_k$ and thus
\begin{align}\label{4.53}
\sup_k\|\varphi_k\|_{1,2;B_\rho(0)}<\infty.
\end{align}
By the definition of the $\varphi_k$, for $M\geq M_0$ sufficiently large and independent of $k$ (note that the sequence $\Theta_k$ is bounded) we have on $B_\rho(0)\cap\big[\lambda_k|\hess u_k|>M\big]$
\begin{align*}
\varphi_k\geq \frac{1}{2}\lambda_k^{-1}\big(\lambda_k|\hess u_k|\big)^{1-\frac{\mu}{2}}=\frac{1}{2}\lambda_k^{-\frac{\mu}{2}}|\hess u_k|^{1-\frac{\mu}{2}}
\end{align*}
and therefore
\begin{align*}
|\hess u_k|^2\leq (2\varphi_k)^{\frac{4}{2-\mu}}\lambda_k^{\frac{2\mu}{2-\mu}}\text{ on }B_\rho(0)\cap\big[\lambda_k|\hess u_k|>M\big].
\end{align*}
According to (\ref{4.53}), $\varphi_k^{\frac{4}{2-\mu}}$ is uniformly integrable, so that
\begin{align*}
\intop_{B_\rho(0)\cap\big[\lambda_k|\hess u_k|>M\big]}|\hess u_k|^2\dz\rightarrow 0\text{ for }k\rightarrow\infty.
\end{align*}
From $\lambda_k|\hess u_k|\rightarrow 0$ a.e. (cf. (\ref{4.44})) it follows
\begin{align*}
\lim_{k\rightarrow\infty}\intop_{B_\rho(0)\cap\big[\lambda_k|\hess u_k|>M\big]}|\hess\hat{u}|\dz=0,
\end{align*}
and (\ref{4.51}) implies $\hess w_k\rightarrow 0$ in $L^2_{\text{loc}}(B_1(0))$, i.e. (\ref{4.49}) holds which proves part $(c)$ of Theorem \ref{Thm1.4} as soon as (\ref{4.50}) and (\ref{4.52}) have been established, what is to follow next.

\noindent\textbf{\textit{ad (\ref{4.52})}}. We return to the Euler equation (\ref{4.19}) choosing $v:=\eta^6\palpha(\udel-P)$, where $P$ denotes a polynomial of degree $\leq 2$ and $\eta$ is specified in the same way as after (\ref{4.19}). We obtain
\begin{align*}
 &\intop_{B_{2r}(x_0)}\eta^6D^2F_\delta(\nabla^2\udel)\big(\palpha\hess\udel,\palpha\hess\udel\big)\dx\\
 &=-\intop_{B_{2r}(x_0)}D^2F_\delta(\nabla^2\udel)\big(\palpha\hess\udel,\nabla^2\eta^6 \palpha[\udel-P]+2\nabla\eta^6\otimes\nabla\palpha(\udel-P)\big)\dx\\
 &\hspace{6,25cm}+\lambda\intop_{B_{2r}(x_0)-D}(\udel-f)\palpha\big(\eta^6\palpha(\udel-P)\big)\dx.
\end{align*}
From (\ref{4.15}) it follows $\udel\in L^\infty_\text{loc}(\Omega)$ uniformly. Applying the Cauchy-Schwarz inequality to the first integral on the right-hand side and using the boundedness of $D^2F$ we get the following estimate, which corresponds to inequality (2.11) in \cite{AF}:
\begin{align*}
 &\intop_{B_{2r}(x_0)}\eta^6D^2F_\delta(\nabla^2\udel)\big(\palpha\hess\udel,\palpha\hess\udel\big)\dx\\
 &\leq c\Bigg\{r^{-4}\intop_{B_{2r}(x_0)}|\nabla(\udel-P)|^2\dx+r^{-2}\intop_{B_{2r}(x_0)}|\nabla^2(\udel-P)|^2\dx\\
 &\hspace{4.5cm}+\intop_{B_{2r}(x_0)}|\nabla^2(\udel-P)|\dx+r^{-1}\intop_{B_{2r}(x_0)}\big|\nabla(\udel-P)\big|\dx\Bigg\}\\
 &\leq c\Bigg\{r^{-4}\intop_{B_{2r}(x_0)}|\nabla(\udel-P)|^2\dx+r^{-2}\intop_{B_{2r}(x_0)}|\nabla^2(\udel-P)|^2\dx+r^2\Bigg\},
\end{align*}
where we have applied Young's inequality. We pass to the limit $\delta\rightarrow 0$ (recall the convergences stated in front of (\ref{4.37})) and get by the boundedness of $D^2F$ and lower semicontinuity 
\begin{align*}
 \intop_{B_r(x_0)}&|\nabla \varphi|^2 \dx\\
 &\leq c\left\{ r^{-4}\intop_{B_{2r}(x_0)}|\nabla u-\nabla P|^2 \dx+r^{-2}\intop_{B_{2r}(x_0)}|\hess u-\hess P|^2\dx+r^2\right\}.
\end{align*}
Now, if $\rho\in (0,1)$ and if we choose in the beginning $\eta\equiv 1$ on $B_{\rho r}(x_0)$, $\text{spt}\,\eta\subset B_r(x_0)$, etc., then it is clear that in place of the latter inequality we obtain
\begin{align}\label{4.50'}
\begin{split}
 &\intop_{B_{\rho r}(x_0)}|\nabla\varphi|^2 \dx\\
 &\leq c(\rho)\left\{r^{-4}\intop_{B_r(x_0)}|\nabla u-\nabla P|^2 \dx+r^{-2}\intop_{B_r(x_0)}|\hess u-\hess P|^2 \dx +r^2\right\}.
\end{split}
\end{align}
Since 
\begin{align*}
 \intop_{B_\rho(0)}|\nabla\varphi_k|^2 \dz =\lambda_k^{-2}\intop_{B_{\rho r_k}(x_k)}|\nabla \varphi|^2 \dx,
\end{align*}
and $\lambda_k^{-2}r_k^2\rightarrow 0$ it is now immediate that (\ref{4.52}) is a consequence of (\ref{4.50'}) and the properties of the sequence $u_k$ provided we choose the polynomial $P$ as done in the definition of the functions $u_k$, i.e.\label{pk}
\begin{align*}
 P_k(z):=r_kA_k\cdot z-a_k-\frac{1}{2}r_k^2\Theta_k(z,z)+\frac{1}{2}r_k^2\mint \limits_{B_1(0)}\Theta_k(y,y)dy.
\end{align*}

\noindent\textbf{\textit{ad (\ref{4.50})} }. Take a cut-off function $\eta\in C_0^\infty(B_1(0))$, $0\leq\eta\leq 1$. Then a Taylor expansion yields
\begin{align}\label{4.54}
\begin{rightcases}
&\lambda_k^{-2}\intop_{B_1(0)}\eta  \bigg[F(\Theta_k+\lambda_k\nabla^2 u_k)-F(\Theta_k+\lambda_k\hess\hat{u})\bigg]\dz\\
&\hspace{4,55cm}-\lambda_k^{-1}\intop_{B_1(0)}\eta DF(\Theta_k+\lambda_k\nabla^2\hat{u}):\hess w_k \dz\\
&\hspace{-0,8cm}=\intop_{B_1(0)}\intop_{0}^1\eta D^2F\big(\Theta_k+\lambda_k\hess\hat{u}+s\lambda_k\hess w_k\big)(\hess w_k,\hess w_k)(1-s)\,ds\,dz\\
\end{rightcases}
\end{align}
\begin{align*}
 &\underset{(\ref{1.15})}{\geq} c\intop_{B_1(0)}\intop_{0}^1\eta\big(1+|\Theta_k+\lambda_k\nabla^2\hat{u}+s\lambda_k\hess w_k|\big)^{-\mu}|\hess w_k|^2(1-s) \,ds\,dz.
\end{align*}

If we just consider the integral from $s=0$ to $s=\frac{1}{2}$ instead of  the whole interval $[0,1]$ in the latter term, it follows
\begin{align}
 \text{r.h.s. of (\ref{4.54})}\geq c\intop_{B_1(0)}\eta \big(1+|\Theta_k|+\lambda_k|\nabla^2\hat{u}|+\lambda_k|\hess w_k|^2\big)^{-\mu}|\hess w_k|^2\dz
\end{align}
and (\ref{4.50}) will follow once we have shown that the left-hand side of (\ref{4.54}) converges to zero as $k\rightarrow\infty$. Note that obviously the left-hand side of (\ref{4.54}) is nonnegative and hence it suffices to give a suitable upper bound.
To this purpose we make use of the convexity of $F$ and obtain
\begin{align*}
&\eta F(\Theta_k+\lambda_k\hess u_k)-\eta F(\Theta_k+\lambda_k\hess \hat{u})\\
\,\\
&=\eta F(\Theta_k+\lambda_k\hess u_k)-\Big\{\eta F(\Theta_k+\lambda_k\hess\hat{u})+(1-\eta)F\big(\Theta_k+\lambda_k\hess u_k\big)\Big\}\\
&\hspace{9cm}+(1-\eta)F(\Theta_k+\lambda_k\hess u_k)\\
\,\\
&\leq \eta F(\Theta_k+\lambda_k\hess u_k)+(1-\eta)F(\Theta_k+\lambda_k\hess u_k)\\
&\hspace{7,1cm}-F(\Theta_k+\eta\lambda_k\hess\hat{u}+(1-\eta)\lambda_k\hess u_k)\\
\,\\
&=F(\Theta_k+\lambda_k \hess u_k)-F\left(\Theta_k+\lambda_k\left[\eta \hess \hat{u}+(1-\eta) \hess u_k\right]\right),
\end{align*}
hence it follows
\begin{align}\label{4.55}
\begin{rightcases}
&\text{l.h.s. of (\ref{4.54})}\\
&\leq \lambda_k^{-2}\intop_{B_1(0)}F(\Theta_k+\lambda_k\hess u_k)\dz\\
&-\lambda_k^{-2}\intop_{B_1(0)}F\left(\Theta_k+\lambda_k\left[\eta \hess\hat{u}+(1-\eta)\hess u_k\right]\right) \dz\\
&-\lambda_k^{-1}\intop_{B_1(0)}\eta DF(\Theta_k+\lambda_k\hess\hat{u}):\hess w_k \dz=: I_1-I_2-I_3.
\end{rightcases}
\end{align}

The minimality of $u$ implies
\begin{align}\label{4.46'}
\begin{split}
 &I_1=\lambda_k^{-2}\intop_{B_1(0)}F\big(\hess u(r_kz+x_k)\big)\dz=\lambda_k^{-2}r_k^{-2}\intop_{B_{r_k}(x_k)}F(\hess u)\dx\\
 &\leq \lambda_k^{-2}r_k^{-2}\Bigg\{\intop_{B_{r_k}(x_k)}F(\hess v)\dx+\frac{\lambda}{2}\intop_{B_{r_k}(x_k)-D}|v-f|^2 \dx\\
 &\hspace{5cm}-\frac{\lambda}{2}\intop_{B_{r_k}(x_k)-D}|u-f|^2 \dx\Bigg\}
\end{split}
\end{align}
for all $v\in W^{2,1}(\Omega)\cap L^2(\Omega-D)$ with $\text{spt}\,(u-v)\Subset B_{r_k}(x_k)$. Now set
\begin{align*}
 v_k(z):=u_k(z)+\eta(z)(\hat{u}-u_k)(z),\;z\in B_1(0),
\end{align*}
and define
\begin{align*}
 \tilde{v}_k(z):=\lambda_k r_k^{2}v_k(z)+P_k(z)
\end{align*}
with $P_k(z)$ as introduced after (\ref{4.50'}). Finally we let
\begin{align*}
 \hat{v}_k(x):=\tilde{v}_k\left(\frac{x-x_k}{r_k}\right)
\end{align*}
with the result that $\text{spt}\,\big(u-\tilv\big)\Subset B_{r_k}(x_k)$ as well as
\begin{align*}
\hess \tilv(x)&=r_k^{-2}\hess \tilde{v}_k\left(\frac{x-x_k}{r_k}\right)\\
&=\lambda_k\hess v_k\left(\frac{x-x_k}{r_k}\right)+\nabla^2P_k\left(\frac{x-x_k}{r_k}\right)r_k^{-2}\\
&=\lambda_k\hess\Big\{u_k+\eta(\hat{u}-u_k)\Big\}\left(\frac{x-x_k}{r_k}\right)+\Theta_k,
\end{align*}
which means that we have
\begin{align*}
 &\lambda_k^{-2} r_k^{-2}\intop_{B_{r_k}(x_k)}F(\hess \tilv) \dx\\
 &=\lambda_k^{-2} \intop_{B_1(0)}F\big(\Theta_k+\lambda_k \hess\big\{u_k+\eta(\hat{u}-u_k)\big\}\big) \dz.
\end{align*}
Going back to (\ref{4.46'}) it follows
\begin{align}\label{4.56}
\begin{split}
 &I_1\leq\overset{\mbox{$=:\widetilde{I}_1$}}{\overbrace{\lambda_k^{-2} \intop_{B_1(0)}F\big(\Theta_k+\lambda_k\hess\big\{u_k+\eta(\hat{u}-u_k)\big\}\big) \dz}}\\
 &\hspace{4cm}+\frac{\lambda}{2}\lambda_k^{-2}r_k^{-2}\intop_{B_{r_k}(x_k)-D}\big(|f-\tilv|^2-|f-u|^2\big)\dx.
\end{split}
\end{align}
Using that $f$ is in $L^\infty(\Omega-D)$, we can estimate the second term on the right-hand side of (\ref{4.56}) by
\begin{align*}
 &\lambda_k^{-2}r_k^{-2}\intop_{B_{r_k}(x_k)-D}\big(|f-\tilv|^2-|f-u|^2\big)\dx\\
 &\leq c \lambda_k^{-2}r_k^{-2}\intop_{B_{r_k}(x_k)}|\tilv -u|\underset{\mbox{bounded!}}{\underbrace{(1+|u|+|\tilv|)}}\dx\\
 &\leq c\lambda_k^{-2}r_k^{-2}\intop_{B_{r_k}(x_k)}|\tilv-u| \dx\\
 &=c\lambda_k^{-2}\intop_{B_1(0)}|\tilde{v}_k(z)-u(x_k+r_kz)| \dz\\
 &=c\lambda_k^{-2}\intop_{B_1(0)}|\lambda_k r_k^2v_k(z)+P_k(z)-u(x_k+r_kz)| \dz\\
 &=c\lambda_k^{-1}r_k^2\intop_{B_1(0)}\bigg|v_k(z)-r_k^{-2}\lambda_k^{-1}\bigg(\big(u(x_k+r_kz)-P_k(z)\big)\bigg)\bigg| \dz\\
 &=c\lambda_k^{-1}r_k^2\intop_{B_1(0)}|v_k-u_k| \dz\\
 &=c\lambda_k^{-1}r_k^2\intop_{B_1(0)}\eta|\hat{u}-u_k| \dz\\
 &\leq c\lambda_k^{-1}r_k^2=:\eps_k\rightarrow 0 \text{ as }k\rightarrow\infty,
\end{align*}
where the last estimate uses (\ref{4.43}). The above estimate combined with (\ref{4.55}) and (\ref{4.56}) now implies
\begin{align}\label{4.57}
 \text{l.h.s. of (\ref{4.54})}\leq \widetilde{I}_1-I_2-I_3+\frac{\lambda}{2}\eps_k.
\end{align}
Starting from (\ref{4.57}), we can follow the arguments in \cite{AF}, p. 209, very closely: let
\begin{align*}
 X_k:=\Theta_k+\lambda_k\big[(1-\eta)\hess u_k+\eta \hess\hat{u}\big],\;Z_k:=2\nabla\eta \otimes\nabla(\hat{u}-u_k)+\hess\eta (\hat{u}-u_k)
\end{align*}
and observe that by another Taylor expansion we have
\begin{align*}
 \widetilde{I}_1-I_2&=\lambda_k^{-1}\intop_{B_1(0)}DF(X_k):Z_k \dz\\
 &+\intop_{B_1(0)}\intop_0^1 D^2F\big(X_k+s\lambda_k Z_k\big)(Z_k,Z_k)(1-s)\,ds\,dz\\
 &\leq \lambda_k^{-1}\intop_{B_1(0)}DF(X_k):Z_k \dz +c\intop_{B_1(0)}|Z_k|^2 \dz,
\end{align*}
 since $D^2 F$ is bounded. By the convergence $u_k\rightarrow \hat{u}$ in $W^{1,2}(B_1(0))$ (recall (\ref{4.43})) we find that $c\intop_{B_1(0)}|Z_k|^2 \dz$ converges to zero and (\ref{4.57}) implies
\begin{align}\label{4.58}
\begin{split}
&\begin{rightcases}
&\text{l.h.s. of (\ref{4.54})}\\
&\leq \lambda_k^{-1}\intop_{B_1(0)}DF(X_k):Z_k \dz\\
&-\lambda_k^{-1}\intop_{B_1(0)}\eta DF(\Theta_k+\lambda_k\hess\hat{u}):\hess w_k \,dz+\eps_k
\end{rightcases}
\end{split}
\end{align}
\begin{align*}
&=\lambda_k^{-1}\intop_{B_1(0)}\big(DF(X_k)-DF\big(\Theta_k+\lambda_k\hess\hat{u}\big)\big):Z_k \dz\\
&-\lambda_k^{-1}\intop_{B_1(0)}DF(\Theta_k+\lambda_k\hess\hat{u}):\hess(\eta w_k)\dz+\eps_k\\
&=:\lambda_k^{-1}I_4-\lambda_k^{-1}I_5+\eps_k
\end{align*}
for another appropriate sequence $\eps_k\rightarrow 0$. With the same idea as in \cite{AF}, p. 211, the integral $I_4$ can be estimated by
\begin{align*}
&\lambda_k^{-1}I_4=\intop_{B_1(0)}\intop_0^1D^2F(\Theta_k+\lambda_k\hess\hat{u}+s\lambda_k(1-\eta)\hess w_k)(\hess w_k, Z_k)(1-\eta) \,ds\,dz\\
&\underset{\text{Hölder}}{\leq} c\underset{\mbox{bounded by (\ref{4.43})}}{\underbrace{\|\hess w_k\|_{2;B_1(0)}}}\;\underset{\hspace{0,5cm}\mbox{$\rightarrow 0$ by (\ref{4.43})}}{\underbrace{\|Z_k\|_{2;B_1(0)}}},
\end{align*}
moreover we have
\begin{align*}
\lambda_k^{-1}I_5&=\lambda_k^{-1}\intop_{B_1(0)}\intop_0^1D^2F\big(\Theta_k+s\lambda_k\hess\hat{u}\big)\big(\lambda_k\hess\hat{u},\hess(\eta w_k)\big) \,ds\,dz\\
&=\intop_{B_1(0)}\intop_0^1 D^2F\big(\Theta_k+s\lambda_k\hess\hat{u}\big)\big(\hess\hat{u},\hess(\eta w_k)\big) \,ds\,dz\\
&\longrightarrow 0\text{ for } k\rightarrow\infty,
\end{align*}
according to (\ref{4.43}) and consequently
\begin{align*}
\text{r.h.s. of (\ref{4.58})}\longrightarrow 0,\;k\rightarrow\infty,
\end{align*}
which finally proves $(\ref{4.50})$.\\
To prove the assertion concerning the Hausdorff-dimension of the singular set we notice that we have (see Remark \ref{rembl} (a))
\begin{align*}
&\Omega-\Omega_0=\left\{x\in\Omega\,:\,\limsup_{r\rightarrow 0}|(\hess u)_{x,r}|=\infty\right\}\\
&\hspace{6cm}\cup\left\{x\in\Omega\,:\,\liminf_{r\rightarrow 0}E(x,r)>0\right\}=:\Omega_1\cup\Omega_2.
\end{align*}
Since $\hess u\in W^{1,s}_{\text{loc}}(\Omega)$ for all $ s<2$ (compare the calculations after (\ref{4.16})), it follows from \cite{Gia}, Theorem 2.1, p. 100, that
\begin{align*}
 \mathcal{H}^{2-s+\kappa}(\Omega_1)=0\quad\forall \kappa>0
\end{align*}
and hence $\mathcal{H}^\eps(\Omega_1)=0$ for $\eps>0$ arbitrarily small. Further (using the Sobolev-Poincaré inequality) it holds
\begin{align*}
 E(x,r)^\frac{1}{2}\leq cr\mint\limits_{B_r(x)}|\nabla^3 u|dy\leq c\left(r^{s-2}\intop_{B_r(x)}|\nabla^3 u|^s dy\right)^\frac{1}{s},
\end{align*}
and \cite{Gia}, Theorem 2.2, p. 101, implies that  $|\nabla^3 u|\in L^s_{\text{loc}}(\Omega)$ together with the above estimate yields $\mathcal{H}^{2-s}(\Omega_2)=0$ and thus $\mathcal{H}^\eps(\Omega_2)=0$ for any $\eps>0$ which concludes the proof of Theorem \ref{Thm1.4} (c). It remains to give the 

\noindent\textbf{\textit{proof of Theorem \ref{Thm1.4} part (d)}}: Let $\sigma_\delta:=DF_\delta(\hess \udel)$. By part (b) of Lemma \ref{Lem41} we have
\begin{align*}
 \palpha \sigma_\delta:\palpha \sigma_\delta&=D^2F_\delta(\hess\udel)\big(\palpha \hess\udel,\palpha\sigma_\delta\big)\\
 &\leq D^2F_\delta(\hess\udel)\big(\palpha\hess\udel,\palpha\hess\udel\big)^\frac{1}{2}D^2F_\delta(\hess\udel)\big(\palpha\sigma_\delta,\palpha\sigma_\delta\big)^\frac{1}{2}\\
 \end{align*}
so that $|\nabla\sigma_\delta|^2\leq \Phi_\delta \big|D^2F_\delta(\hess\udel)\big|^\frac{1}{2}|\nabla\sigma_\delta|$ where we have abbreviated 
\begin{align*}
 \Phi_\delta:=D^2F_\delta(\hess\udel)\big(\palpha\hess\udel,\palpha\hess\udel\big)^\frac{1}{2}.
\end{align*}
According to inequality (\ref{est428}) we know
\begin{align*}
 \Phi_\delta\in L^2_{\text{loc}}(\Omega)\quad\text{uniformly in }\delta.
\end{align*}
By the condition of $\mu$-ellipticity (\ref{1.15}) it holds $\big|D^2F_\delta(Z)\big|\leq c(1+\delta)$ and thus
\begin{align*}
 |\nabla\sigma_\delta|\leq c(1+\delta)\Phi_\delta
\end{align*}
which implies $\sigma_\delta\in W^{1,2}_{\text{loc}}(\Omega)\quad\text{uniformly in }\delta$, and in particular $\sigma_\delta\rightharpoondown:\overline{\sigma}$ in $W^{1,2}_\text{loc}(\Omega)$. But then it follows $\overline{\sigma}=\sigma$, since (due to $u_\delta\in W^{3,s}_\text{loc}(\Omega)$ uniformly for $s<2$) $\nabla^2\udel\rightarrow\nabla^2 u$ a.e., which finishes the proof of part (d) and thereby the proof of Theorem \ref{Thm1.4}. \qed
\end{section}

\vspace{1cm}

\begin{tabular}{l l}
Martin Fuchs & \hspace{2cm}Jan Müller\\
Saarland University & \hspace{2cm}Saarland University\\
Department of Mathematics & \hspace{2cm}Department of Mathematics\\
P.O. Box 15 11 50 & \hspace{2cm}P.O. Box 15 11 50\\
66041 Saarbrücken & \hspace{2cm}66041 Saarbrücken\\
Germany & \hspace{2cm}Germany\\
fuchs@math.uni-sb.de & \hspace{2cm}jmueller@math.uni-sb.de
\end{tabular}
\end{document}